\documentclass{amsart}

\usepackage[utf8]{inputenc}
\usepackage{lmodern}
\usepackage[english]{babel}
\usepackage{amsmath}
\usepackage{amsfonts}
\usepackage{amssymb}
\usepackage{graphicx}
\usepackage{tabularx}
\usepackage{subcaption}
\newcommand{\curl}{{curl}}
\newcommand{\dive}{{div}}

\newcommand{\R}{ {\mathbb{R}} }

\newcommand{\N}{ {\mathbb{N}} }
\newcommand{\SSS}{ {\mathbb{S}} }

\newcommand{\id}{ {\operatorname{id}} }
\newcommand{\bosy}[1]{{\boldsymbol{#1}}}

\newcommand{\ee}{{\bosy{e}}}
\newcommand{\ff}{{\bosy{f}}}
\newcommand{\gggg}{{\bosy{g}}}

\newcommand{\nn}{{\bosy{n}}}

\newcommand{\sss}{{\bosy{s}}}
\newcommand{\uu}{{\bosy{u}}}
\newcommand{\vv}{{\bosy{v}}}
\newcommand{\ww}{{\bosy{w}}}
\newcommand{\xx}{{\bosy{x}}}
\newcommand{\yy}{{\bosy{y}}}
\newcommand{\zz}{{\bosy{z}}}

\newcommand{\HH}{{\bosy{H}}}

\newcommand{\LL}{{\bosy{L}}}

\newcommand{\aalpha}{{\boldsymbol{\alpha}}}
\newcommand{\bbeta}{{\boldsymbol{\beta}}}
\newcommand{\ggamma}{{\boldsymbol{\gamma}}}

\newcommand{\ttheta}{{\boldsymbol{\theta}}}

\newcommand{\nnu}{{\boldsymbol{\nu}}}

\newcommand{\rrho}{{\boldsymbol{\rho}}}

\newcommand{\DDelta}{{\boldsymbol{\Delta}}}
\newcommand{\cX}{\mathcal{X}}

\newcommand{\cC}{\mathcal{C}}
\newcommand{\cR}{\mathcal{R}}
\newcommand{\cP}{\mathcal{P}}
\newcommand{\cB}{\mathcal{B}}
\newcommand{\cI}{\mathcal{I}}

\newcommand{\cF}{\mathcal{F}}
\newcommand{\cN}{\mathcal{N}}
\newcommand{\bcF}{\bosy{\cF}}

\newcommand{\veps}{\varepsilon }

\DeclareMathOperator*{\esssup}{ess\,sup}

\newtheorem{theorem}{Theorem}[section]
\newtheorem{lemma}[theorem]{Lemma}

\begin{document}

\title[collocation method]{An analytically divergence-free collocation method for the incompressible Navier-Stokes equations on the rotating sphere}

\author{Tino Franz}
\address{Applied and Numerical Analysis, University of Bayreuth, D-95440 Bayreuth, Germany}
\email{Tino.Franz@uni-bayreuth.de, Orcid-ID:0000-0002-4273-7639}
\date{\today}

\thanks{The author would like to thank Holger Wendland for his encouragement and several valuable discussions which lead to this paper.
} 
\subjclass[MSC Classification]{65M12, 65M70, 76D05}
\keywords{collocation, radial basis function, meshfree methods, Navier-Stokes equation, vector spherical harmonics}

\begin{abstract}
In this work, we develop a high-order collocation method using radial basis function (RBF) for the incompressible Navier-Stokes equation (NSE) on the rotating sphere. 
The method is based on solving the projection of the NSE on the space of divergence-free functions. 
For that, we use matrix valued kernel functions which allow an analytically divergence-free approximation of the velocity field.  
Using kernel functions which lead to rotation-free approximations, the pressure can be recovered by a simple kernel exchange in one of the occurring approximations, without solving an additional Poisson problem.
We establish precise error estimates for the velocity and the pressure functions for the semi-discretised solution. 
In the end, we give a short estimate of the numerical cost and apply the new method to an experimental test case.
\end{abstract}

\maketitle
\section{Introduction}
\label{intro}
In this paper we will develop and analyse a discretisation method for the incompressible Navier-Stokes equations (NSE) on the rotating unit sphere using a kernel-based divergence-free approximation method. 
This method is highly motivated by and can be seen as an extension of the work of Keim \cite{Keim2016_Diss}, who constructed a collocation scheme for the time-depended Navier-Stokes equations on the $d$-dimension torus. 
See also \cite{Keim2016} for the method applied to the time-dependent Stokes equation.

Kernel-based divergence-free approximation methods have been studied and analysed in \cite{Dodu2002,Benbourhim2010,Fuselier2008,Fuselier2008_2,Fuselier2016,Lowitzsch2005,Lowitzsch2005_2,Farrell2017}, and were successfully applied to different problems, see for example, \cite{Schrader2011,Schrader2012,Wendland2009}.
In \cite{Ward2007} and \cite{fuselier2009}, the authors extended and analysed the theory of divergence-free and rotation-free kernel-based approximations for tangential vector fields on manifolds, especially the unit sphere. 
These approximation methods will be used here.

The two dimensional sphere in $\R^3$ is given by $\SSS^{2}=\{\xx\in\R^3\,\vert\,\|\xx\|_2=1\}$, where $\|\cdot\|_2$ denotes the standard euclidean norm. 
A vector field $\uu:\SSS^{2}\to\R^3$ is called to be tangential, if $\uu(\xx)\cdot\xx=0$ for all $\xx\in\SSS^2$.
The Navier-Stokes equations on the sphere are given as follows, see \cite{Il1988,Titi1999}.
Given a tangential and divergence-free vector field $\uu_0:\SSS^2\to\R^3$ and a tangential and divergence-free forcing term $\ff:[0,T)\times\SSS^2\to\R^3$, we seek the solution  $\uu:[0,T)\times  \SSS^2\to\R^3$, which represents the velocity, and $p:[0,T)\times \SSS^2\to\R$, which represents the pressure, of
\begin{align}\label{Introduction::NS::1}
\partial_t\uu-\nu\DDelta^*\uu+\cB(\uu,\uu)+\cC(\uu)+
\nabla^* p&=\ff&&\text{on } \SSS^{2}\times(0,T),\\\label{Introduction::NS::2}
\nabla^*\cdot\uu&=0&&\text{on } \SSS^{2}\times(0,T),\\\label{Introduction::NS::3}
\uu(\cdot,0)&=\uu_0&&\text{on } \SSS^{2},
\end{align}
where 
\begin{equation*}
\cB(\uu,\vv):=(\uu\cdot\nabla_*)\vv,\quad \cC(\uu):=l\nn\times\uu,
\end{equation*}
$\nu>0$ is the diffusion coefficient, $\nn$ is the normal vector and $l(\xx)=2\Omega x_3=2\Omega \sin(\phi)$ is the Coriolis coefficient with angular velocity $\Omega$ and latitude $\phi=\arcsin(x_3)$. 
$\cB$ represents the non-linear convection operator, while $\cC$ is the Coriolis operator caused by the rotation of the sphere. 
A proper definition of the differential operators on the sphere will be given in Section \ref{notation::basic}.

To simplify this system, a common approach is to project equation \eqref{Introduction::NS::1} on its divergence-free part using the Leray operator $\cP$ to eliminate the gradient of the pressure, which leads to 
\begin{equation}\label{Introduction::NS_projected::1}
\partial_t\uu-\nu\DDelta^*\uu=\cP(\ff-\cB(\uu,\uu)-\cC(\uu)),\quad\text{in }(0,T)\times \SSS^{2},
\end{equation} 
where $\uu$ still satisfies \eqref{Introduction::NS::2}.

Our method is using a kernel-based interpolation method, which automatically gives an approximation of the divergence-free as well as the rotation-free part of a function.
More precisely, given a set of points $\cX=\{\xx_1,\ldots\,\xx_N\}\subset \SSS^{2}$ and data sites $\ff_j=\ff(\xx_j)\in \R^3$, $1\leq j\leq N$, with $\ff$ being a tangential vector field, the interpolation of $\ff$ is given by
\begin{equation}\label{Introduction::NS_approx::1}
\ff_h(\xx)=\sum_{j=1}^N\Phi(\xx,\xx_j)\bbeta_j,\quad \xx\in \SSS^3,
\end{equation}
where $\bbeta_j\in\R^3$ with $\bbeta_j\cdot\xx_j=0$ for every $1\leq j\leq N$ and the kernel $\Phi:\SSS^{2}\times\SSS^{2}\to \R^{3\times3}$ is given such that the interpolation $\ff_h$ is tangential.
In our context, the kernel will be given such that $\Phi=\Phi_\dive+\Phi_\curl$, i.e. the kernel function is the sum of a kernel function which yields divergence-free and one that yields curl-free interpolations.
Hence, by simply exchanging the kernel in \eqref{Introduction::NS_approx::1} by $\Phi_\dive$, we obtain an analytically divergence-free approximation to the divergence-free part of $\ff$, which leads to an approximation $\cP_h$ of the Leray operator. 
See \cite{Ward2007} and \cite{fuselier2009} for more details.

To derive an approximation of the divergence-free velocity, we use the method of lines to separate the variables, such that our approximation has the form
\begin{equation}\label{Introduction::velocity_approx}
\uu_h(\xx,t)=\sum_{j=1}^N\Phi_\dive(\xx,\xx_j)\aalpha_j(t),\quad \xx\in \SSS^3, t\in(0,T).
\end{equation}
This ensures that \eqref{Introduction::NS::2} is automatically satisfied.
Equation \eqref{Introduction::NS_approx::1} is then to be enforced pointwise on $\cX$, which leads to the semi-discrete system of ODEs
\begin{equation}\label{Introduction::NS_approx}
\partial_t\uu_h(\xx_j,t)-\nu\DDelta^*\uu_h(\xx_j,t) =\cP_h(\ff-\cB(\uu_h,\uu_h)-\cC(\uu_h))(\xx_j,t)\
\end{equation}
for $\xx_j\in \cX$ and $t\in(0,T).$
The main part of the work consists of analysing the convergence of the solution of \eqref{Introduction::NS_approx} to the solution of \eqref{Introduction::NS_projected::1}. 
Here, the right hand side is given by using the approximated Leray operator $\cP_h$, which gives a divergence-free approximation of the form
\begin{equation}\label{Introduction::rhs_approx}
\cP_h(\ff-\cB(\uu_h,\uu_h)-\cC(\uu_h))(\xx_j,t)=\sum_{k=1}^N\Phi_\dive(\xx_k,\xx_j)\bbeta_j(t)
\end{equation}
for $\xx_j\in \cX$ and $t\in(0,T)$, where the coefficients $\bbeta_j\in\R^3$ are obtained by interpolating $\ff-\cB(\uu_h,\uu_h)-\cC(\uu_h)$ on the supporting points $\cX$. 

Using standard time discretisation methods such as Runge-Kutta methods, the system given in \eqref{Introduction::NS_approx} can be turned into a fully discretised system. 
Hence, the linear system to determine $\bbeta_j$ for \eqref{Introduction::rhs_approx} as well as the linear system to determine the $\aalpha_j$ in \eqref{Introduction::velocity_approx} have to be solved in every time-step.
However, choosing kernel functions $\Phi$ with compact support, the matrices in both linear systems can be made sparse, such that both linear systems can be solved sufficiently fast.
As we will additionally point out, the efficiency can be theoretically further increased by using Lagrange basis functions.

This method has several advantages. 
It provides analytically divergent-free approximations to the velocity, which do not need any spatial numerical integration for their calculation.
Moreover, it can be set up to construct arbitrarily smooth solution just by the choice of the kernel function, which leads to high-order spatial approximations.

Finally, it leads easily to a high-order approximation of the pressure just by exchanging the kernel function in \eqref{Introduction::rhs_approx}, i.e.
\begin{equation*}\label{Introduction::pressure_approx}
p_h(\xx,t)=\sum_{j=1}^N\Psi(\xx,\xx_j)^T\bbeta_j(t).
\end{equation*}
for a specific kernel function $\Psi:\SSS^2\times\SSS^2\to\R^3$. 
Hence, we do not need to solve an additional Poisson equation to obtain an approximation of the pressure.
However, we have to assume a strong solution to the Navier-Stokes equations to ensure the existence of the point evaluations for the collocation method.

Even if there is a lot of work for numerical methods for the Navier-Stokes equations on the Euclidean space or bounded Euclidean domains, see for example \cite{Debussche1993,temam2001}, the are only few results about numerical methods and their analysis for the Navier-Stokes equations on bounded manifolds and the sphere.
First of all, there is the paper of Fengler and Freeden \cite{fengler2004}, who developed a spectral non-linear Galerkin method which is based on vector and tensor spherical harmonics. 
However, they mainly concentrate on the implementation of their method, and did not provide a proper mathematical analysis.
Next, there is the paper of Ganesh, Le Gia and Sloan \cite{Ganesh2010}, who developed a pseudospectral quadrature method for the Navier-Stokes equations. 
Based on the Gevrey-regularity, they were able to show spectral convergence of their method. However, their method needs numerical integration of the occurring integrals.
Moreover, both method do not provide a fast and easy possibility to calculate an approximation of the pressure.

This paper is organized as follows. 
In the second section we will give a short introduction to spherical harmonics and Sobolev spaces on the sphere. 
The approximation tools we will need for constructing the our approximation method will be given in the third section.
In the fourth section we will shortly introduce the problem we want to solve, together with a few theoretical results concerning the non-linear part of the Navier-Stokes equations. 
After that, in the fifth section, we will construct our approximation scheme and give the proof of convergence as the main result of this article.
Finally, we demonstrate the new method numerically.

\section{Notation and preliminaries}\label{notation}
In this section, we want to provide the basic theory to spherical harmonics and function spaces on the sphere. 
We will start with basic notations.
\subsection{Basic notations and differential operators on the sphere}\label{notation::basic}
We will differ between scalar-valued functions $f:\SSS^{2}\to \R$ and vector fields $\uu:\SSS^{2}\to \R^3$, which will be written in bold letters.
Hence, we will also denote spaces, which contain vector-valued functions, with bold letters.
For a $\xx\in\SSS^{2}$, the set $T_\xx\SSS^{2}$ contains all $\aalpha\in\R^3$ which are tangent to $\SSS^{2}$ at $\xx$, i.e. $\aalpha\cdot\xx=0$.

We define the surface gradient $\nabla^*$, the surface curl-operator $\boldsymbol{L}^*$ and the surface Laplace-Beltrami operator $\Delta^*$ for a sufficiently smooth scalar-valued function $f:\SSS^{2}\to \R$ by
\begin{align*}
\nabla^* f&:=\nabla f-\left<\nabla f,\nnu\right>\nnu,\\
\LL^*f(\xx)&:=\xx\times\nabla^*f(\xx),\\
\Delta^* f(\xx)&:=\Delta f(\xx)-\left<\xx,H(f)\xx\right>-2\left<\nabla f(\xx),\xx\right>,
\end{align*}
where $\nabla$ and $\Delta$ are the standard gradient and Laplace operators, respectively, $\times$ denotes the vector product and $H$ denotes the Hessian Matrix $H(f)=(\partial_{ij}f)$.
For $1\leq j\leq 3$, the $j$-th component of $\nabla^*$ will be denoted by $\partial^*_j$. 
The Laplace-Beltrami operator can also be written as
\begin{equation*}
\Delta^*=\nabla^*\cdot\nabla^*=\LL^*\cdot\LL^*.
\end{equation*}
For vector fields $\uu:\SSS^{2}\to \R^3$, we define the surface-divergence $\nabla^*\cdot$, the vectorial curl-operator $\LL^*\cdot$ and the vectorial Laplace-Beltrami operator $\DDelta^*$ as
\begin{align*}
\nabla^*\cdot\uu&:=\sum_{i=1}^3 (\nabla^*u_i)\cdot\ee_i,\\
\LL^*\cdot \uu(\xx)&:=\sum_{i=1}^3 (\LL^*u_i)\cdot\ee_i,\\
\DDelta^* &:=p_{nor}(\Delta^*+2)p_{nor}+p_{tan}\Delta^*p_{tan},
\end{align*}
where $\ee_i$ denotes the $i$-th unit vector, $p_{nor}(\uu)(\xx)=\langle \uu(\xx),\xx\rangle\xx$ and $p_{tan}(\uu)=\uu-p_{nor}\uu$ are the normal and tangential part of the function $\uu$ with respect to the sphere, respectively. 
For further details on differential operators on the sphere, see for example \cite{Freeden1998}.
Note that we will only be interested in tangential vector fields $\uu$, where the normal part vanishes. 
In this case, the vectorial Laplace-Beltrami operator degenerates to the scalar one applied to the single components $u_i$, $1\leq i\leq3$, of $\uu$. 
Finally, we will call a vector field $\uu$ to be divergence-free if $\nabla^*\cdot\uu$ vanishes and curl-free if the rotation $\LL^*\cdot\uu$ vanishes.

We will use time-depended function spaces, which are defined as follows.
For an arbitrary Banach space $X$, the space $C([0,T];\!\! X)$ consists of all continuous functions from $[0,T]$ to $X$.
For $1\leq p\leq \infty$, we denote the space of all strongly measurable functions from $[0,T]$ to $X$ such that $t\mapsto\|f(t)\|_X\in L^p(0,T)$ by $L^p(0,T;\!\!X)$, with norm 
\begin{equation*}
\|f\|_{L^p(0,T;\!\!X)}=
\begin{cases}
\left(\int_0^T\|f(t)\|_X^2dt\right)^{1/p},& 1\leq p<\infty,\\
\esssup_{0\leq t\leq T}\|f(t)\|_X,& p=\infty.
\end{cases}
\end{equation*}
For an $\uu:\SSS^2\times[0,T]\to\R^3$, we will sometimes write $\uu(t)$ instead of $\uu(\cdot,t)$ for the sake of readability.

\subsection{Spherical Harmonics and the space of square integrable functions}
As usual, the Hilbert space of square integrable functions on the sphere is given by
\begin{equation*}
L^2:=L^2(\SSS^{2})=\left\{
f:\SSS^2\to\R\,
\Big\vert\,\int_{\SSS^2}
\vert f(\xx)\vert^2
dS(\xx)
<\infty
\right\}
\end{equation*}
where the inner product is given by
\begin{equation*}
\langle f,g\rangle  :=\int_{S^{2}}f(\xx)g(\xx)dS(\xx),\quad f,g\in L^2.
\end{equation*}
Another way to describe $L^2$ functions on the sphere is by using spherical harmonics, see \cite{mueller1966,Freeden1998} for an introduction.
The spherical harmonics $Y_{\ell,k}$ of order $\ell\in\N_0$ and $1\leq k\leq 2\ell+1$ are the eigenfunctions of $-\Delta_*$ with corresponding eigenvalues
\begin{equation*}
\lambda_\ell=\ell(\ell+1),\quad \ell\in\N_0.
\end{equation*}
They form an orthonormal basis of $L^2$, which means that every function $f\in L^2$ has a Fourier representation
\begin{equation*}
f(\xx)=\sum_{\ell=0}^\infty\sum_{k=1}^{2n+1 }\widehat{f}_{\ell,k}Y_{\ell,k}(\xx),\quad \widehat{f}_{\ell,k}=\langle f,Y_{\ell,k}\rangle.
\end{equation*}
Using Parseval's identity, the $L^2$ inner product can then also be written in terms of the Fourier coefficients by
\begin{equation*}
\langle f,g\rangle=\sum_{\ell=0}^\infty\sum_{k=1}^{2n+1 }\widehat{f}_{\ell,k}\widehat{g}_{\ell,k},\quad f,g\in L^2.
\end{equation*}
The idea can be extended to vector-valued functions. We denote the space of square-integrable vector-valued functions which are tangential to the sphere by
\begin{equation*}
\LL^2:=\LL^2(\SSS^{2})=\left\{\uu:\SSS^2\to\R^3\,\Big\vert\,\uu\text{ tangential and }\int_{\SSS^2}\|\uu(\xx)\|_2^2dS(\xx)<\infty\right\}
\end{equation*}
where the inner product is given by
\begin{equation*}
\langle \uu,\vv\rangle  :=\int_{S^{2}}\uu(\xx)\cdot\vv(\xx)dS(\xx),\quad \uu,\vv\in \LL^2.
\end{equation*}
Again, we can build an orthonormal basis for this function space which is given by the analogue of the spherical harmonics for vector fields: the vectorial spherical harmonics. 
There are three different, $\LL_2$-orthogonal types: 
the divergence-free and the curl-free type, which are both tangential to the sphere, and one type that is normal to the sphere. 
Here, we are only interested in the tangential type.

For $\ell\geq 1$ and $1\leq k\leq 2\ell+1$ the divergence-free vector spherical harmonics are given by $\yy_{\ell,k}=\LL_*Y_{\ell,k}/\sqrt{\lambda_\ell}$ and the curl-free vector spherical harmonics are given by $\zz_{\ell,k}=\nabla^*Y_{\ell,k}/\sqrt{\lambda_\ell}$ for all . Together, they form an orthonormal basis for $\LL^2$ with respect to the $\LL^2$ norm. Moreover, $\yy_{\ell,k}$ and $\zz_{\ell,k}$ are the eigenfunctions to the vectorial Laplace Beltrami operator with corresponding eigenvalue $\lambda_\ell$. 

Again, we can write each vector field $\uu\in \LL^2$ in its Fourier form
\begin{equation*}
\uu=\sum_{\ell=1}^\infty\sum_{k=1}^{2\ell+1}\left(\widehat{\uu}^\dive_{\ell,k}\yy_{\ell,k}+\widehat{\uu}^\curl_{\ell,k}\zz_{\ell,k}\right),
\end{equation*}
where the divergence-free and the curl-free Fourier coefficients are given by
\begin{equation*}
\widehat{\uu}^\dive_{\ell,k}=\langle\uu, \yy_{\ell,k}\rangle,\quad \widehat{\uu}^\curl_{\ell,k}=\langle\uu, \zz_{\ell,k}\rangle.
\end{equation*}
The $\LL^2$ inner product can again be written as $\langle \uu,\vv\rangle=\sum_{\ell=1}^\infty\sum_{k=1}^{2n+1 }\widehat{\uu}^\dive_{\ell,k}\widehat{\vv}^\dive_{\ell,k}+\widehat{\uu}^\curl_{\ell,k}\widehat{\vv}^\curl_{\ell,k}$ for $\uu,\vv\in \LL^2$.

The Fourier expansion gives us a direct and unique decomposition of a tangential vector field $\uu$ in a divergence-free part $\uu_\dive=\sum_{\ell=0}^\infty\sum_{k=1}^{2\ell+1}\widehat{\uu}^\dive_{\ell,k}\yy_{\ell,k}$ and a curl-free part $\uu_\curl=\sum_{\ell=0}^\infty\sum_{k=1}^{2\ell+1}\widehat{\uu}^\curl_{\ell,k}\zz_{\ell,k}$.
The corresponding Hilbert spaces for square-integrable divergence-free and the curl-free functions are therefore given by
\begin{equation*}
\begin{split}
\LL^2_\dive&:=\left\{ \uu\in\LL^2\,\vert\,\widehat{\uu}^\curl_{\ell,k}=0,\;1\leq \ell\leq \infty,1\leq k\leq 2\ell+1  \right\},
\\
\LL^2_\curl&:=\left\{ \uu\in\LL^2\,\vert\,\widehat{\uu}^\dive_{\ell,k}=0,\;1\leq \ell\leq \infty,1\leq k\leq 2\ell+1 \right\}
\end{split}
\end{equation*}
with given inner products
\begin{equation*}
\begin{split}
\langle \uu,\vv\rangle_{\LL^2_\dive}
&:=\sum_{\ell=0}^\infty\sum_{k=1}^{2\ell+1}\widehat{\uu}^\dive_{\ell,k}\widehat{\vv}^\dive_{\ell,k}
=\int_{\SSS^2}\uu_\dive(\xx)\vv_\dive(\xx)dS(\xx),
\quad \uu,\vv\in \LL^2_\dive,
\\
\langle \uu,\vv\rangle_{\LL^2_\curl}
&:=\sum_{\ell=0}^\infty\sum_{k=1}^{2\ell+1}\widehat{\uu}^\curl_{\ell,k}\widehat{\vv}^\curl_{\ell,k}
=\int_{\SSS^2}\uu_\curl(\xx)\vv_\curl(\xx)dS(\xx),
\quad \uu,\vv\in \LL^2_\curl.
\end{split}
\end{equation*}
for $\uu,\vv\in \LL^2$. 
The spaces $\LL^2_\dive$ and $\LL^2_\curl$ are a direct decomposition of $\LL^2$, i.e. $ \LL^2=\LL^2_\dive\oplus\LL^2_\curl$.
Note that for sufficiently smooth $\uu\in\LL^2_\dive$ and $\vv\in\LL^2_\curl$ we have $\nabla^*\cdot\uu=0$ and $\LL^*\cdot \vv=0$. 
Finally, we define the Leray operator $\cP:\LL^2\to\LL^2_\dive$, $\uu\mapsto\uu_\dive$, which maps to the divergence-free part of a vector field $\uu\in\LL^2$. The operator, which maps to the curl-free part of a vector field is then simply given by $(id-\cP):\LL^2\to\LL^2_\curl$.
\subsection{Sobolev spaces on the sphere}
Sobolev spaces on the sphere can simply be introduced in terms of (vector) spherical harmonics. For $\sigma>0$, we define the Sobolev spaces for scalar valued functions on the sphere by $H^\sigma:=\{f\in L^2\,\big\vert\,\|f\|_{H^\sigma}<\infty\}$,  where the norm is induced by the inner product, which is given by
\begin{equation*}
\langle f,g\rangle_{H^\sigma}:=\sum_{\ell=0}^\infty(1+\lambda_\ell)^\sigma\sum_{k=1}^{N(d,\ell)}\widehat{f}_{\ell,k}\widehat{g}_{\ell,k}.
\end{equation*}
In the vector-valued case, we define the divergence-free and the curl-free Sobolev spaces by
\begin{equation*}
\begin{split}
\HH^\sigma_\dive
&=\left\{\uu\in\LL^2_\dive\,\big\vert\,\|\uu\|_{\HH^\sigma_\dive}<\infty \right\},\quad
\HH^\sigma_\curl
=\left\{\uu\in\LL^2_\curl\,\big\vert\,\|\uu\|_{\HH^\sigma_\curl}<\infty \right\},
\end{split}
\end{equation*}
where the corresponding inner products are given by
\begin{equation*}
\begin{split}
\langle\uu,\vv\rangle_{\HH^\sigma_\dive}&:=\sum_{\ell=1}^\infty(1+\lambda_\ell)^\sigma\sum_{k=1}^{2\ell+1}\widehat{\uu}^\dive_{\ell,k}\widehat{\vv}^\dive_{\ell,k},
\quad \uu,\vv\in\HH^\sigma_\dive, \\
\langle\uu,\vv\rangle_{\HH^\sigma_\curl}&:=\sum_{\ell=1}^\infty(1+\lambda_\ell)^\sigma\sum_{k=1}^{2\ell+1}\widehat{\uu}^\curl_{\ell,k}\widehat{\vv}^\curl_{\ell,k}
\quad \uu,\vv\in\HH^\sigma_\curl .
\end{split}
\end{equation*}
The tangential Sobolev space is then simply given by the direct sum $\HH^\sigma=\HH^\sigma_\dive\oplus\HH^\sigma_\curl$ with inner product $\langle\uu,\vv\rangle_{\HH^\sigma}=\langle\uu,\vv\rangle_{\HH^\sigma_\dive}+\langle\uu,\vv\rangle_{\HH^\sigma_\curl}$ for $\uu,\vv\in\HH^\sigma$.
Finally, note that if $\sigma\in\N_0$, the norm for $\HH^\sigma$ can also be written as
\begin{equation*}
\|\uu\|_{\HH^\sigma}=\|(1-\DDelta^*)^\sigma\uu\|_{\LL^2}.
\end{equation*}
This also applies to the $\HH^\sigma_\dive$ and $\HH^\sigma_\curl$ norms.

\subsection{Auxiliary results}
For the following analysis, we will need some auxiliary results for fuctions from the Sobolev space.
It is easy to see that if $\uu$ is a tangential vector field, then the zero-th Fourier coefficients of its scalar-valued components vanishes, i.e. $\widehat{(u_i)}_{0,0}=0$ for $1\leq i\leq 3$, see \cite[Theorem 12.4.1]{Freeden1998}.
The first result, which uses this fact, is a generalization of \cite[Lemma 3.10]{fengler2004} and gives us a connection between the regularity of a vector field and its components. 
Since the idea of proof does not change, it will be omitted.
\begin{lemma}\label{6::lemma:2}
Let $\sigma\geq 0$ and suppose that $\uu\in\HH^\sigma$. Then, 
\begin{equation*}
\|u_i\|_{H^\sigma}\leq C\|\uu\|_{\HH^\sigma},\quad i=1,2,3.
\end{equation*}
\end{lemma}

The next auxiliary result gives us the regularity of the product of two Sobolev functions. The proof can be found in \cite{agranovic1965} and \cite{skiba2017}.
\begin{lemma}\label{6::lemma:3}
Let $\sigma>1$ and $f,g\in H^\sigma$. Then, $fg\in H^\sigma$ and there exists a constant $C=C(\sigma)>0$ such that
\begin{equation*}
\|fg\|_{H^\sigma}\leq C\|f\|_{H^\sigma}\|g\|_{H^\sigma}.
\end{equation*} 
\end{lemma}
We will need the following extension of Lemma \ref{6::lemma:3} for vector fields.
\begin{lemma}\label{6::lemma:4}
Let $\sigma>1$.
\begin{enumerate}
\item[a)] Let $\uu,\vv\in \HH^\sigma$. Then, $\uu\cdot\vv\in H^\sigma$ and there exists a constant $C=C(\sigma)>0$ such that
\begin{equation*}
\|\uu\cdot\vv\|_{H^\sigma}\leq C\|\uu\|_{\HH^\sigma}\|\vv\|_{\HH^\sigma}.
\end{equation*} 
\item[b)] Let $f\in H^\sigma$ and $\uu\in \HH^\sigma$. Then, $f\uu\in \HH^\sigma$ and there exists a constant $C=C(\sigma)>0$ such that
\begin{equation*}
\|f\uu \|_{\HH^\sigma}\leq C\|f\|_{H^\sigma}\|\uu\|_{\HH^\sigma}.
\end{equation*} 
\end{enumerate}

\end{lemma}
\begin{proof}
Part a) of this lemma follows easily by the triangle inequality, Lemma \ref{6::lemma:2} and Lemma \ref{6::lemma:3}, which together yield
\begin{equation*}
\|\uu\cdot\vv \|_{H^\sigma}
\leq \sum_{i=1}^3\|u_iv_i\|_{H^\sigma}
\leq C\sum_{i=1}^3\|u_i\|_{H^\sigma}\|v_i\|_{H^\sigma}
\leq 3C\|\uu \|_{\HH^\sigma}\|\vv \|_{\HH^\sigma}.
\end{equation*}
To prove the second part, we first note that we can bound the norm of a vector field $\vv\in\HH^\sigma$ by
\begin{equation*}
\|\vv \|_{\HH^\sigma}\leq  C\left(	\|\nabla^*\cdot\vv \|_{H^{\sigma-1}}+\|\LL^*\cdot \vv \|_{H^{\sigma-1}}\right).
\end{equation*}
Setting $\vv=f\uu$ and using the product rule yields
\begin{equation*}
\begin{split}
\|f\uu \|_{\HH^\sigma}
\leq& 
C\big(
\|f\nabla^*\cdot\uu \|_{H^{\sigma-1}}+
\|\nabla^* f\cdot\uu \|_{H^{\sigma-1}}\\
&+
\|f\LL^*\cdot\uu \|_{H^{\sigma-1}}+
\|\LL^* f\cdot\uu \|_{H^{\sigma-1}}\big).
\end{split}
\end{equation*}
Using the first part of this lemma together with Lemma \ref{6::lemma:4} finishes the proof.
\end{proof}
Finally, to derive high-order estimates for the non-linear operator $\cB$, we will need the following result about the norm of partial derivatives of a vector field.
\begin{lemma}\label{6::lemma:1}
Let $\sigma>2$ and $\uu\in\HH^\sigma$. Then, we have
\begin{equation*}
\sum_{j=1}^3\|\partial_j^*\uu\|_{\HH^\sigma}^2
\leq
\|\uu\|_{\HH^{\sigma+1}}^2.
\end{equation*}
This inequality also holds for the $\HH^\sigma_\dive$ and $\HH^\sigma_\curl$ norms.
\end{lemma}
\begin{proof}
We will show the inequality only for the divergence-free part, the proof for the curl-free part is nearly the same.
With $\vv=(1-\DDelta^*)^\sigma\uu$, we can rewrite the norm as 
\begin{equation*}
\|\partial_j^*\uu\|_{\HH^\sigma_\dive}^2
=
\langle\partial_j^*\vv,\partial_j^*\vv\rangle_{\LL^2_\dive}
=
-\langle(\partial_j^*)^2\vv,\vv\rangle_{\LL^2_\dive}
\end{equation*}
Summing over $j$ and using that $\widehat{(\DDelta^*\uu)}^\dive_{\ell,k}=-\lambda_\ell\widehat{\uu}^\dive_{\ell,k}$ then yields
\begin{equation*}
\begin{split}
\sum_{j=1}^3\|\partial_j^*\uu\|_{\HH^\sigma_\dive}^2
=&
-\langle\DDelta^*\vv,\vv\rangle_{\LL^2_\dive}
=
-\langle\DDelta^*\uu,\uu\rangle_{\HH^\sigma_\dive}\\
=&-\sum_{\ell=1}^\infty(1+\lambda_\ell)^\sigma\sum_{k=1}^{2\ell+1}\widehat{(\DDelta^*\uu)}^\dive_{\ell,k}\widehat{\uu}^\dive_{\ell,k}
\\
=&\sum_{\ell=1}^\infty\lambda_\ell(1+\lambda_\ell)^\sigma\sum_{k=1}^{2\ell+1}(\widehat{\uu}^\dive_{\ell,k})^2
\leq
\|\uu\|_{\HH^{\sigma+1}_\dive}^2,
\end{split}
\end{equation*}
which finishes the proof.
\end{proof}

\section{Kernel-Based Interpolation and Native Spaces}
In this section we will give a brief introduction to the tools needed for the kernel-based interpolation of vector fields on the sphere. 
For a general introduction to kernel-based approximation, see, for example, \cite{wendland2004}.
\subsection{Native spaces for scalar valued functions on the sphere}
We start with a spherical basis function (SBF) $\phi:\SSS^2\times\SSS^2\to\R$ of the form
\begin{equation*}
\phi(\xx,\yy)=\sum_{\ell=0}^\infty\sum_{k=1}^{2\ell+1}\widehat{\phi}_\ell Y_{\ell,k}(\xx)Y_{\ell,k}(\yy),\quad \xx,\yy\in\SSS^2,
\end{equation*}
with certain Fourier coefficients $\widehat{\phi}_\ell\in\R$ such that  $\sum_{\ell=0}^\infty(2\ell+1)\widehat{\phi}_\ell<\infty$. This kernel can be seen as the Fourier series of a zonal function, see \cite[Ch. 17]{wendland2004}. 
The kernel $\phi$ is positive definite if $\widehat{\phi}_\ell>0$ for all $\ell\geq 0$. 
In this case the native space of the kernel $\phi$ and its inner product are given by
\begin{equation*}
\cN_{\phi}:=\left\{ f\in L^2\,\Big\vert\,\sum_{\ell=0}^\infty\sum_{k=1}^{2\ell+1}\frac{\widehat{f}_{\ell,k}^{\,2}}{\widehat{\phi}_\ell}<\infty\right\},\quad
\langle f,g\rangle_\phi:=\sum_{\ell=0}^\infty\sum_{k=1}^{2\ell+1}\frac{\widehat{f}_{\ell,k}\widehat{g}_{\ell,k}}{\widehat{\phi}_\ell}.
\end{equation*}
This space is a so-called reproducing kernel Hilbert space (RKHS), which on the one hand means that $\phi(\cdot,\xx)\in\cN_{\phi}$ and on the other hand that $f(\xx)=\langle f,\phi(\cdot,\xx)\rangle_\phi$ for all $f\in\cN_{\phi}$ and $\xx\in\SSS^2$.
It is a well known fact that $\cN_{\phi}=H^\sigma$ if there exists a constant $C>0$ such that $\widehat{\phi}_\ell$ yield a lower and upper bound of the form
\begin{equation}\label{2::equation::bound_for_native_space}
C^{-1}(1+\lambda_\ell)^{-\sigma}
\leq \widehat{\phi}_\ell
\leq C(1+\lambda_\ell)^{-\sigma}, \quad \ell\in\N_0.
\end{equation}
In this is case, the $H^\sigma$-norm and the $\cN_{\phi}$-norm are equivalent.

Spherical basis functions can easily be constructed out of radial basis functions, since they are zonal functions restricted to the sphere, see \cite{fuselier2009,Ward2007}. 
Suitable choices for RBFs are for example the Gaussian or the Wendland radial basis functions, see \cite{wendland1995}. 
Here, we will use the latter one for the construction of our numerical scheme. 
The advantage of these functions is that they can be constructed of arbitrary smoothness and such that they satisfy \eqref{2::equation::bound_for_native_space} for a given $\sigma>1$, which means that their native space is isomorphic to the Sobolev space, see \cite{Narcowich2002}.
Moreover, they are piecewise polynomials with compact support, which makes them efficient to calculate. 
Examples of the Wendland functions of various smoothness which are suitable for the sphere are given in Table \ref{2::table::1}. See also \cite{wendland2004} for more details.

\begin{center}
\begin{table}[t]
\begin{tabular}{lc}
\hline 
Function & $\sigma$\vspace*{1pt} \\ 
\hline \vspace*{1pt}
$\phi_1(r)\doteq(1-r)_+^4(4r+1)$ & $5/2$\vspace*{1pt}  \\ 
$\phi_2(r)\doteq(1-r)_+^6(35r^2+18r+3)$ & $7/2$\vspace*{1pt}  \\ 
$\phi_3(r)\doteq(1-r)_+^8(32r^3+25r^2+8r+1)$ & $9/2$\vspace*{1pt}  \\ 
$\phi_4(r)\doteq(1-r)_+^{10}(429r^4+450r^3+210r^2+50r+5)$ & $11/2$\vspace*{1pt}  \\ 
\hline \hline 
\end{tabular}
\caption{Compactly supported radial basis functions of various smoothness.} 
\label{2::table::1}
\end{table}
\end{center}

\subsection{Kernel-based interpolation for tangential vector spaces}
Positive definite kernels on the sphere which yield interpolants for divergence-free vector-valued functions where introduced in \cite{Ward2007}. 
In \cite{fuselier2009}, positive definite kernels  which yield interpolants for rotation free vector fields were introduced and an error analysis was established. 
Both, the divergence-free as well as the rotation free kernel can be constructed from the scalar-valued kernel $\phi$.

Assuming that $\sum_{\ell=0}^\infty(2\ell+1)\lambda_\ell\widehat{\phi}_\ell<\infty$, the kernel which yields divergence-free interpolants is given by the matrix-valued function
\begin{equation}\label{2::eq::Phi_div_long}
\Phi_\dive(\xx,\yy):=\LL^*_\xx(\LL^*_\yy)^T\phi(\xx,\yy)=\sum_{\ell=1}^\infty\lambda_\ell\widehat{\phi}(\ell)\sum_{k=1}^{2\ell+1}\yy_{\ell,k}(\xx)\yy_{\ell,k}^T(\yy),
\end{equation}
where $\LL^*_\xx$ and $(\LL^*_\yy)^T$ act on the $\xx$ and $\yy$ variables, respectively. It is a well known fact that if $\widehat{\phi}(\ell)>0$ for all $\ell\geq 1$ then $\Phi_\dive$ is positive definite, which means that given a set of pairwise distinct points $\cX=\{\xx_1,\ldots ,\xx_N\}\subset \SSS^{2}$ and an arbitrary set of tangent vectors
$\aalpha_j\in T_{\xx_j}(\SSS^2)\setminus\{\boldsymbol{0}\}$, we have
\begin{equation*}
\sum_{j,k=1}^N\aalpha_j^T\Phi_\dive(\xx_j,\xx_k)\aalpha_k>0.
\end{equation*}
Its curl-free counterpart is given by the matrix-valued kernel
\begin{equation}\label{2::eq::Phi_curl_long}
\Phi_\curl(\xx,\yy):=\nabla_{\xx}^*(\nabla_{\yy}^{*})^T\phi(\xx,\yy)=\sum_{\ell=1}^\infty\lambda_\ell\widehat{\phi}(\ell)\sum_{k=1}^{2\ell+1}\zz_{\ell,k}(\xx)\zz_{\ell,k}^T(\yy).
\end{equation}
Again, if $\widehat{\phi}(\ell)>0$ for all $\ell\geq 1$, then $\Phi_\curl$ is positive definite.
The native space for both kernels $\Phi_\dive$ and $\Phi_\curl$ are given by
\begin{equation*}
\begin{split}
\cN_{\Phi_\dive}
&=
\!\!\left\{\uu\in\LL^2\,\Big\vert\, 
\sum_{\ell=1}^\infty\sum_{k=1}^{2\ell+1}\frac{(\widehat{\uu}^\dive_{\ell,k})^2}{\lambda_\ell\widehat{\phi}(\ell)}
<\infty\right\}, 
\\
\!\cN_{\Phi_\curl}
&=
\!\!\left\{\uu\in\LL^2\,\Big\vert\, 
\sum_{\ell=1}^\infty\sum_{k=1}^{2\ell+1}\frac{(\widehat{\uu}^\curl_{\ell,k})^2}{\lambda_\ell\widehat{\phi}(\ell)}
<\infty\right\},
\end{split}
\end{equation*}
where the norms are induced by the inner products, which are given by
\begin{equation*}
\langle\uu,\vv\rangle_{\Phi_\dive}\!=\!\sum_{\ell=1}^\infty  \sum_{k=1}^{2\ell+1}\frac{1}{\lambda_\ell\widehat{\phi}(\ell)}\widehat{\uu}^\dive_{\ell,k}\widehat{\vv}^\dive_{\ell,k},
\quad\!
\langle\uu,\vv\rangle_{\Phi_\curl}\!=\!\sum_{\ell=1}^\infty  \sum_{k=1}^{2\ell+1}\frac{1}{\lambda_\ell\widehat{\phi}(\ell)}\widehat{\uu}^\curl_{\ell,k}\widehat{\vv}^\curl_{\ell,k}.
\end{equation*}
Both spaces are again RKHS with reproducing kernels $\Phi_\dive$ and $\Phi_\curl$, respectively, which means in the case of $\Phi_\dive$ that for every $\uu\in\cN_{\Phi_\dive}$, $\xx\in\SSS^2$ and $\aalpha\in T_\xx(\SSS^2)$ we have $\Phi_\dive(\cdot,\xx)\aalpha\in\cN_{\Phi_\dive}$ on the one hand and $\aalpha^T\uu(\xx)=\langle \uu,\Phi_\dive(\cdot,\xx)\aalpha\rangle_{\Phi_\dive}$ on the other hand \cite{Fuselier2009_2}.
If the Fourier coefficients behave like $\widehat{\phi}_\ell\sim (1+\lambda_\ell)^{-(\sigma+1)}$ for a $\sigma>0$, then $\cN_{\Phi_\dive}=\HH^{\sigma}_\dive$ and $\cN_{\Phi_\curl}=\HH^{\sigma}_\curl$ with equivalent norms, which means there is a constant $C>0$ with
\begin{equation*}
\begin{split}
C^{-1}\|\uu\|_{\HH^{\sigma}_\dive}
&\leq \|\uu\|_{\Phi_\dive}
\leq C\|\uu\|_{\HH^{\sigma}_\dive},\quad 
\uu\in\HH^{\sigma}_\dive,\\
C^{-1}\|\uu\|_{\HH^{\sigma}_\curl}
&\leq \|\uu\|_{\Phi_\curl}
\leq C\|\uu\|_{\HH^{\sigma}_\curl},\quad 
\uu\in\HH^{\sigma}_\curl.
\end{split}
\end{equation*}
Finally, the kernel $\Phi=\Phi_\dive+\Phi_\curl$ is also positive definite if $\Phi_\dive$ and $\Phi_\curl$ are positive definite.
The native space $\cN_{\Phi}$ of $\Phi$ is then given by the direct sum of the native spaces $\cN_{\Phi_\dive}$ and $\cN_{\Phi_\curl}$, i.e. $\cN_{\Phi}=\cN_{\Phi_\dive}\oplus\cN_{\Phi_\curl}$. Hence,  $\cN_{\Phi}$ can be identified with $\HH^{\sigma}$ and their norms are equivalent. 
The exact representation for $\Phi_\dive$ and $\Phi_\curl$ defined by a zonal kernel $\phi$ can be found in \cite{fuselier2009}.

\subsection{Interpolation with Kernel Functions}\label{2.3}
From now on, we assume that we have a positive definite matrix-valued kernel $\Phi=\Phi_\dive+\Phi_\curl$ with $\widehat{\phi}_\ell\sim (1+\lambda_\ell)^{-(\sigma+1)}$ for a $\sigma>0$.
In the following, we define a discrete set $\cX=\{\xx_1,\ldots ,\xx_N\}\subset \SSS^{2}$ of $N\in\N$ pairwise distinct data points with fill distance
\begin{equation*}
h_\cX:=\sup_{\xx\in\SSS^2}\min_{\xx_j\in\cX}d(\xx,\xx_j),
\end{equation*}
where $d(\xx,\yy)$ is the geodesic distance between $\xx,\yy\in\SSS^2$.
Furthermore, we define our approximation space
\begin{equation*}
\bcF_\Phi(\cX):=\operatorname{span}\{\Phi(\cdot,\xx_j)\aalpha_j\:\vert\:\xx_j\in\cX,\,\aalpha_j\in T_{\xx_j}(\SSS^2),\,1\leq j\leq N\}.
\end{equation*}
The interpolation operator, which maps a function $\uu\in\HH^\sigma$ to its best approximation in $\bcF_\Phi(\cX)$, is given by
\begin{equation*}
\cI_{\Phi}:\HH^\sigma\to \bcF_\Phi(\cX),\quad \uu\mapsto \sss_{\cX,\uu}=\sum_{j=1}^N\Phi(\cdot,\xx_j)\aalpha_j,
\end{equation*}
where the coefficients $\aalpha_j\in T_{\xx_j}(\SSS^2)$ are uniquely determined by $\sss_{\cX,\uu}(\xx_j)=\uu(\xx_j)$ for all $1\leq j\leq N$.

Since the kernel $\Phi$ is just the sum of its divergence-free and its curl-free part, we also have a simple approximation for the divergence-free and the curl-free part of $\uu$. Note that the operator $\cI_{\Phi}$  as well as the following operators depend on the data sites $\cX$. However, since we do not vary the data sites, we will omit this dependency due to readability.

Projecting an interpolant onto the space of divergence-free functions, we can use that $\cP\Phi(\xx_j,\cdot)\aalpha_j=\Phi_\dive(\cdot,\xx_j)\aalpha_j$. Hence, we define the interpolation operator of the divergence-free part, which can also be seen as an approximated Leray operator, by
\begin{equation}\label{2.3::Leray_approx}
\cP_{\Phi}:\HH^\sigma\to \bcF_{\Phi_\dive}(\cX),\quad \uu\mapsto \sum_{j=1}^N\Phi_\dive(\cdot,\xx_j)\aalpha_j.
\end{equation}
In the same way, we can project an interpolant on its curl-free part by using the operator
\begin{equation*}\label{2.3::inv_Leray_approx}
(\operatorname{id}-\cP) \cI_{\Phi}:\HH^\sigma\to \bcF_{\Phi_\curl}(\cX),\quad \uu\mapsto \sum_{j=1}^N\Phi_\curl(\cdot,\xx_j)\aalpha_j.
\end{equation*}
From \cite{fuselier2009}, we will use the following error estimates for the three interpolation operators, which are given in terms of the fill distance $h_\cX$.
\begin{theorem}\label{2::Theorem::1}
Let $\sigma>3/2$ and $\uu\in\HH^{\sigma}$.
Then, there exists a $C>0$ such that for all $\tau<\sigma$
\begin{equation*}
\|\uu-\cI_{\Phi}\uu\|_{\HH^{\tau}}\leq C h_\cX^{\sigma-\tau}\|\uu\|_{\HH^\sigma}.
\end{equation*}
In particular, there also exist constants $C_1,C_2>0$ such that 
\begin{equation*}
\begin{split}
\| \cP\uu- \cP_{\Phi} \uu \|_{\HH^\tau} 
&\leq C_1 h_\cX^{\sigma-\tau} \|\uu\|_{\HH^\sigma}\\
\| (\id-\cP)\uu- (\id-\cP_{\Phi}) \uu \|_{\HH^\tau} 
&\leq C_2 h_\cX^{\sigma-\tau} \|\uu\|_{\HH^\sigma}
\end{split}
\end{equation*}
These estimates also hold in the $\HH^\sigma_\dive$ and $\HH^\sigma_\curl$ norms. 
\end{theorem}
\subsection{Collocation methods for Helmholtz Equation}
For the convergence result, we will need an approximation method for the Helmholtz equation.
For a given right hand side $ \ff:\SSS^{2}\to\R^3$ the vectorial Helmholtz equation is given by
\begin{equation}\label{eq::Helmholtz}
-\DDelta^*\uu+\uu=\ff,\quad\text{on } \SSS^{2}.
\end{equation}
Its solution can easily be obtained via Fourier analysis. 
Every right hand side $\ff$ with zero divergence admits a divergence-free solution $\uu$.

Let $\widehat{\phi}_\ell\sim (1+\lambda_\ell)^{-(\sigma+1)}$ with $\sigma>5/2$ and $\sss_{\cX,\uu}=\sum_{j=1}^N\Phi_\dive(\cdot,\xx_j)\aalpha_j\in \bcF_{\Phi_\dive}$. 
We want to find the coefficient-vectors $\aalpha_j\in T_{\xx_j}(\SSS^2)$, $1\leq j\leq N$, such that $\sss_{\cX,\uu}$ is an approximation in $\bcF_{\Phi_\dive}$ to $\uu\in\HH^{\sigma+1}_\dive\subset\HH^\sigma_\dive$ with the property that $(-\DDelta^*+\operatorname{id})\sss_{\cX,\uu}(\xx_j)=(-\DDelta^*+\operatorname{id})\uu(\xx_j)$ for every $\xx_j\in \cX$.

Using \eqref{2::eq::Phi_div_long} and the fact that $\DDelta^*\yy_{\ell,m}=-\lambda_\ell\yy_{\ell,m}$, we easily see that
\begin{equation*}
\left(-\DDelta^*_\xx+\operatorname{id}_\xx\right)\Phi_\dive(\xx,\yy)
=
\sum_{\ell=1}^\infty\lambda_\ell(1+\lambda_\ell)\widehat{\phi}(\ell)\sum_{k=1}^{2\ell+1}\yy_{\ell,m}(\xx)\yy_{\ell,m}^T(\yy).
\end{equation*}
If $\Phi_\dive$ is positive definite, so is $\left(-\DDelta^*+\operatorname{id}\right)\Phi_\dive$, and we can find uniquely determined $\aalpha_j\in  T_{\xx_j}(\SSS^2)$, $1\leq j\leq N$, with
\begin{equation}\label{2::equation::Ritz::1}
\left(-\DDelta^*+\operatorname{id}\right)\sss_{\cX,\uu}(\xx_j)=\ff(\xx_j)
=\left(-\DDelta^*+\operatorname{id}\right)\uu(\xx_j)
,\quad \xx_j\in\cX.
\end{equation} 
The operator
\begin{equation*}
\cR_{\Phi_\dive}:\HH^{\sigma+1}_\dive\to \bcF_{\Phi_\dive},\quad
 \uu\to \sss_{\cX,\uu},
\end{equation*}
 which gives the approximation $\sss_{\cX,\uu}$ to the solution $\uu$ of the Helmholtz-equation, is called Ritz-projection operator.
\begin{theorem}\label{approximation::theroem::2}
Let $\sigma>5/2$. If $\uu\in \HH^{\sigma+1}_\dive$ then there exists exactly one approximation $\sss_{\cX,\uu}=\cR_{\Phi_\dive}\uu\in \bcF_{\Phi_\dive}$ which satisfies $(-\DDelta^*+\operatorname{id})\sss_{\cX,\uu}(\xx_j)=(-\DDelta^*+\operatorname{id})\uu(\xx_j)$ for every $\xx_j\in \cX$.
Moreover, there exists a constant $C>0$ such that
\begin{equation*}
\|\uu-\cR_{\Phi_\dive}\uu\|_{\HH^{\tau_1}}\leq C h_\cX^{\tau_2-\tau_1}\|\uu\|_{\HH^\tau_2} ,\quad \uu\in \HH^{\tau_2}_\dive
\end{equation*}
for all $5/2\leq \tau_1\leq\sigma+1\leq\tau_2\leq 2\sigma$.
\end{theorem}
\begin{proof}
As we can easily verify, if $\ff\in \HH^{\sigma}_\dive$ and $\uu$ is a solution of \eqref{eq::Helmholtz}, then $\uu\in\HH^{\sigma+2}_\dive$ and we have
\begin{equation*}
\|\uu\|_{\HH^{\sigma+2}_\dive}=\|\ff\|_{\HH^{\sigma}_\dive}.
\end{equation*}
Hence , we can apply Theorem \ref{2::Theorem::1}, which yields
\begin{equation*}
\begin{split}
\|\uu-\cR_{\Phi_\dive}\uu\|_{\HH^{\tau_1}}
&=
\|\ff-\cI_{(-\DDelta^*+\operatorname{id})\Phi_\dive,\cX}\ff\|_{\HH^{\tau_1-2}}\\
&\leq
C h^{(\tau_2-2)-(\tau_1-2)}\|\ff\|_{\HH^{\tau_2-2}}\\
&=
C h^{\tau_2-\tau_1}\|\uu\|_{\HH^{\tau_2}}.
\end{split}
\end{equation*}
The uniquiness of the solution comes from the fact that the kernel $\left(-\DDelta^*_\xx+\operatorname{id}_\xx\right)\Phi_\dive$ is positive definite.
\end{proof}

\section{The Navier-Stokes equations on the Sphere and auxiliary Results}
As mentioned in the introduction, we want to solve the incompressible Navier-Stokes equations on the rotating sphere \eqref{Introduction::NS::1} - \eqref{Introduction::NS::3}.
We do this by applying the Leary operator $\cP$ to equation \eqref{Introduction::NS::1}. 
By \eqref{Introduction::NS::2}, the velocity itself is divergence-free, so the time derivative of the velocity and the diffusion term are not affected, but the gradient of the pressure vanishes. We arrive at
\begin{align}\label{equation::NS_divfree::1}
\partial_t\uu-\nu\Delta_*\uu&=\cP \left(\ff-\cB(\uu,\uu)-\cC(\uu)\right)&&\text{in } \SSS^{2}\times(0,T),\\
\label{equation::NS_divfree::2}
\uu(\cdot,0)&=\uu_0&&\text{on } \SSS^{2}.
\end{align}
Hence, the velocity can be calculated by finding a solution of \eqref{equation::NS_divfree::1} - \eqref{equation::NS_divfree::2}.
On the other hand, applying the operator $\cI-\cP$ to \eqref{Introduction::NS::1} yields the equation
\begin{equation}\label{equation::NS_curlfree}
\nabla_* p=(id-\cP) \left(\ff-\cB(\uu,\uu)-\cC(\uu)\right)\quad\text{in }\SSS^{2}\times(0,T),
\end{equation} 
from which the pressure can be calculated. 
Since the pressure can only be calculated up to a constant, we will assume the pressure to have mean zero integral, i.e. $\int_{\SSS^2}p(\xx)dS(\xx)=0$. 
The pair $(\uu,p)$ which will be given by the solutions of \eqref{equation::NS_divfree::1} and \eqref{equation::NS_curlfree} is the a solution of the Navier-Stokes equations \eqref{Introduction::NS::1} - \eqref{Introduction::NS::3}.

The existence and uniqueness of the solution of the divergence-free part of the Navier-Stokes equations has been discussed in \cite{Il1988,Il1990,Il1993}. 
However, the results are often based on the weak formulation and do not yield smooth solutions.
Since we need high-order regularity of our solution, we will use the result of \cite{Titi1999}, which is based on the Gevrey regularity. The following lemma is a simple deduction of their existence and uniqueness result.
\begin{lemma}\label{3::Lemma::existence}
Let $\sigma>0$. Assume that $\uu_0\in\HH^{2\sigma+1}_\dive$ and $\ff\in L^\infty(0,\infty;\LL^2_\dive)$ such that 
$$\sum_{\ell=1}^\infty\sum_{k=1}^{2\ell+1}\lambda_\ell^{2\sigma}e^{\tau\lambda_\ell^{1/2}}(\widehat{\ff(\cdot,t)}_{\ell,k}^\dive)^2<\infty$$ 
for some $\tau>0$ and all $t\in(0,\infty)$. Then, \eqref{equation::NS_divfree::1} - \eqref{equation::NS_divfree::2} has a unique global solution $u\in L^\infty(0,\infty;\HH^{2\sigma+1}_\dive)$.
\end{lemma}
Hence, this result ensures that there exists a solution for the divergence-free part of the Navier-Stokes equations \eqref{equation::NS_divfree::1} - \eqref{equation::NS_divfree::2}. The existence of the solution of \eqref{equation::NS_curlfree} is simply given by the fact that the right hand side of \eqref{equation::NS_curlfree} is curl-free.

Next we want to gain further information about the two operators $\cB$ and $\cC$.
As one can easily see, $\cB:\HH^\sigma\times\HH^{\sigma+1}\to\HH^\sigma$ is a bilinear operator. 
%Unfortunately, $\cB(\uu,\vv)$ has not to be a tangential vectorfield any more, even if $\uu$ and $\vv$ are both tangential. 
%However, since we are only interested in the divergence-free and curl-free projection of the Navier-Stokes equation, we can neglect this fact.
In the following lemma we will show that the operator is also bounded.
\begin{lemma}\label{auxiliary_results::lemma::1}
Let $\sigma>1$. Then, there exist a constant $C_1>0$, which only depends on $d$ and $\sigma$, such that
\begin{equation*}
\|\cB(\uu,\vv)\|_{\HH^\sigma}
\leq 
C_1\|\uu\|_{\HH^\sigma}\left(\sum_{j=1}^3\|\partial_j^*\vv\|_{\HH^{\sigma}}^2\right)^{1/2}
\leq
C_1\|\uu\|_{\HH^\sigma}\|\vv\|_{\HH^{\sigma+1}}, \quad ,
\end{equation*}
for all $\uu\in \HH^{\sigma}$ and $\vv\in \HH^{\sigma+1}$. Moreover, there exist a constant  $C_2>0$, which depends on $d$ and $\sigma$, such that
\begin{equation*}
\langle\cB(\uu,\vv),\ww\rangle_{\HH^\sigma}
\leq 
C_2\|\uu\|_{\HH^{\sigma-1}}\|\vv\|_{\HH^{\sigma}}\|\ww\|_{\HH^{\sigma+1}}.
\end{equation*}
for all $\uu\in\HH^{\sigma-1}$, $\vv\in \HH^{\sigma}$ and $\ww\in \HH^{\sigma+1}$.
\end{lemma}
\begin{proof}
For the first term, the triangle inequality, Lemma \ref{6::lemma:3} and the Cauchy-Schwarz inequality yield
\begin{equation*}
\|\cB(\uu,\vv)\|_{\HH^\sigma}
\leq\sum_{i=1}^3\|u_i\partial^*_i\vv\|_{\HH^\sigma}
\leq C\left(\sum_{i=1}^3\|u_i\|^2_{H^\sigma}\right)^{1/2}\left(\sum_{i=1}^3\|\partial^*_i\vv\|^2_{\HH^\sigma}\right)^{1/2}.
\end{equation*}
Using Lemma \ref{6::lemma:2} gives the first inequality of the first statement. The second one follows from Lemma \ref{6::lemma:1}.
For the second estimate, we use that 
\begin{equation*}
\langle\cB(\uu,\vv),\ww\rangle_{\HH^\sigma}\leq \|\cB(\uu,\vv)\|_{\HH^{\sigma-1}}\|\ww\|_{\HH^{\sigma+1}}.
\end{equation*}
Applying the first estimate finishes the proof.
\end{proof}
It is also easy to see that $\cC:\HH^\sigma\to\HH^\sigma$ is a bounded linear operator. This will be given in the following lemma.
\begin{lemma}\label{auxiliary_results::lemma::2}
Let $\sigma>1$ and $\uu\in \HH^\sigma$. Then, there exists a constant $C>0$, which only depends on $d$ and $\sigma$ such that
\begin{equation*}
\begin{split}
\|\cC(\uu)\|_{\HH^\sigma}&\leq 
C\Omega\|\uu\|_{\HH^\sigma}.
\end{split}
\end{equation*}
\end{lemma}
\begin{proof}
Using Lemma \ref{6::lemma:4}, we have 
\begin{equation*}
\|\cC(\uu)\|_{\HH^\sigma}=2\Omega\|n_3\nn\times\uu\|_{\HH^\sigma}\leq 2\Omega\|n_3\|_{H^\sigma}\|\nn\times\uu\|_{\HH^\sigma}.
\end{equation*}
Moreover, using that $\nn\times \yy_{\ell,k}=\zz_{\ell,k}$ for $\ell\in\N$ and $1\leq k\leq 2\ell+1$, the divergence-free Fourier coefficients of $\nn\times \uu$ satisfy
\begin{equation*}
\begin{split}
(\widehat{\nn\times \uu})_{\ell,k}^\dive
=&
\int_{\SSS^2}\left(\nn\times \uu\right)(\nn)\cdot\yy_{\ell,k}(\xx)dS(\xx)\\
=&\int_{\SSS^2}\uu(\xx)\cdot \left(\nn\times \yy_{\ell,k}\right)(\xx)dS(\xx)
=\widehat{\uu}_{\ell,k}^\curl. 
\end{split}
\end{equation*}
With the same argument, one can also show that $(\widehat{\nn\times \uu})_{\ell,k}^\curl=\widehat{\uu}_{\ell,k}^\dive$ which means that $\|\nn\times\uu\|_{\HH^\sigma}=\|\uu\|_{\HH^\sigma}$.
\end{proof}

\section{The Semi-discrete Problem}
We now want to give the semi-discretised collocation scheme for the incompressible  Navier-Stokes equations on the sphere.
We will differ between the divergence-free part to get an approximation of the velocity and the curl-free part to get an approximation of the pressure.
\subsection{Approximation Scheme for the Velocity}
As mentioned in the introduction, the idea is to approximate the velocity with the interpolation scheme given in section \ref{2.3} and using the method of lines by
\begin{equation}\label{equation::approx::1}
\uu_h(\cdot,t)=\sum_{j=1}^N\Phi_\dive(\cdot,\xx_j)\aalpha_j(t),\quad t\in[0,T_h),
\end{equation}
where $T_h>0$ is the maximal time of existence of $\uu_h$.
By its construction, the approximation is divergence-free and automatically satisfies equation \eqref{Introduction::NS::2}.
Moreover, to get an approximation of the right hand side in \eqref{equation::NS_divfree::1}, we exchange the Leray operator by its approximated version from \eqref{2.3::Leray_approx}. 
Then, we arrive at the following problem: We search for the solution $\uu_h\in C^1([0,\infty);\cF_{\Phi_\dive}(X))$ which satisfies the discretised equations
\begin{align}
\partial_t\uu_h-\nu\Delta^*\uu_h
&=\cP_{\Phi}\left(\ff-\cB(\uu_h,\uu_h)-\cC(\uu_h)\right),&&\text{in }\cX\times(0,T_h),\label{equation::NS_2::1}\\
\uu_h(\cdot,0)&=\uu_0,&&\text{in } \cX.\label{equation::NS_2::2}
\end{align}
To be more precise, we now have a system of ODEs in time, where we have to find the coefficient vectors $\aalpha_j\in C^1([0,\infty);T_{\xx_j}(\SSS^2))$, $j=1,\ldots, N$. 
Standard theory of ordinary differential equation ensures the local-in-time existence of a solution up to a maximum time which we will denote by $T_h$. 
This time depends on the initial data, the right hand side as well as the data sites $\cX$. However, analogue to the continuation criterion for ODEs, one can show that the solution of \eqref{equation::NS_2::1} - \eqref{equation::NS_2::2} has to exist globally in time if the solution is bounded on the maximum time-interval $[0,T_h)$.

We will investigate the difference between the solution $\uu$ of the projected Navier-Stokes equations \eqref{equation::NS_divfree::1} - \eqref{equation::NS_divfree::2} and our approximated solution $\uu_h$ of \eqref{equation::NS_2::1} - \eqref{equation::NS_2::2}
\begin{equation*}
\|\uu(\cdot,t)-\uu_h(\cdot,t)\|_{\HH^\sigma_\dive}\leq \|\rrho_h(\cdot,t)\|_{\HH^\sigma_\dive}+\|\ttheta_h(\cdot,t)\|_{\HH^\sigma_\dive},
\end{equation*}
where we split up the error into the projection error $\rrho_h=\uu-\cR_{\Phi_\dive}\uu$ and the stability error $\ttheta_h=\cR_{\Phi_\dive}\uu-\uu_h$. 
The projection error can easily be estimated by Theorem \ref{approximation::theroem::2}, so it remains to investigate the stability error.
The procedure will be as follows. To find an estimate for the stability error $\ttheta_h$, we will investigate
\begin{equation}\label{Semi::eq::procedure}
\begin{split}
\frac{1}{2}\partial_t\|\ttheta_h\|^2_{\Phi_\dive}
-\langle\nu\boldsymbol{\Delta}^*\ttheta_h,\ttheta_h\rangle_{\Phi_\dive}
%\nu\|\LL^*\cdot\ttheta_h\|_{\phi}
=&
\left\langle -\partial_t\rrho_h+\nu\rrho_h+\ff-\cI_{\Phi_\dive}\ff, \ttheta_h\right\rangle_{\Phi_\dive}\\
&-\left\langle \cB(\uu,\uu)-\cI_{\Phi_\dive}\cB(\uu_h,\uu_h) ,\ttheta_h\right\rangle_{\Phi_\dive}\\
&-\left\langle \cC(\uu)-\cI_{\Phi_\dive}\cC(\uu_h),\ttheta_h\right\rangle_{\Phi_\dive}.
\end{split}
\end{equation}
The identity \eqref{Semi::eq::procedure} will be shown in Lemma \ref{Semi::lemma::1}. Estimates for the last two terms of the right hand side will be given separately in Lemma \ref{Semi::lemma::2} and Lemma \ref{Semi::lemma::3}. 
After that, we will give a local-in-time error estimate for the approximated solution in Lemma \ref{Convergence::lemma::1}, before we will use a bootstrap argument to show that the estimate is valid on the whole time interval where the analytical solution exists. 
Theorem \ref{4::theroem::main_theorem} will be the main convergence result of this paper.
Note that, according to Lemma \ref{3::Lemma::existence}, the analytical solution will exist global in time under certain conditions. 
In this case we will show that the approximated solution will exist global in time as well.

Note that in \eqref{Semi::eq::procedure} the term $-\nu\langle\boldsymbol{\Delta}^*\ttheta_h,\ttheta_h\rangle_{\Phi_\dive}$ remains positive since
\begin{equation*}
\begin{split}
-\nu\langle\boldsymbol{\Delta}^*\ttheta_h,\ttheta_h\rangle_{\Phi_\dive}
&=
-\sum_{\ell=1}^\infty\frac{1}{\lambda_\ell\widehat{\phi}_\ell}\sum_{k=1}^{2\ell+1}
(\ttheta_h)_{\ell,k}^\dive
\int_{\SSS^2}\DDelta^*\ttheta_h(\xx)\cdot\yy_{\ell,k}(\xx)dS(\xx)\\
%\int_{\SSS^2}\ttheta_h(\xx)\cdot\yy_{\ell,k}(\xx)dS(\xx)\\
&=
\sum_{\ell=1}^\infty\frac{1}{\widehat{\phi}_\ell}\sum_{k=1}^{2\ell+1}
(\ttheta_h)_{\ell,k}^\dive
\int_{\SSS^2}\ttheta_h(\xx)\cdot\frac{-\DDelta^*\yy_{\ell,k}(\xx)}{\lambda_\ell}dS(\xx)\\
%\int_{\SSS^2}\ttheta_h(\xx)\cdot\yy_{\ell,k}(\xx)dS(\xx)\\
&=
\sum_{\ell=1}^\infty\frac{1}{\widehat{\phi}_\ell}\sum_{k=1}^{2\ell+1}
\left(\int_{\SSS^2}\ttheta_h(\xx)\cdot\yy_{\ell,k}(\xx)dS(\xx)\right)^2
>0.
\end{split}
\end{equation*}
Unfortunately, the term on the right hand side is none of our native space norms of the kernel $\phi$. 
However, it can define a norm which is equivalent to the $\HH^{\sigma+1}_\dive$ norm. 
Hence, we can find a constant $C>0$ such that
\begin{equation}\label{4::eq::equivalece_inner_product}
C^{-1}\|\vv\|_{\HH^{\sigma+1}_\dive}^2
\leq-\langle\boldsymbol{\Delta}^*\vv,\vv\rangle_{\Phi_\dive}
\leq C\|\vv\|_{\HH^{\sigma+1}_\dive}^2
\end{equation}
for all $\vv\in\HH^{\sigma+1}_\dive$.
\begin{lemma}\label{Semi::lemma::1}
Let $\sigma>1$, $\ff\in L^\infty(0,T_h;\HH^{\sigma})$ and $\uu\in C^1([0,T_h),\HH^{\sigma+1})$. For every test function $\chi\in \cF_{\Phi_\dive}(\cX)$, the equation
\begin{equation*}
\begin{split}
\langle \partial_t\ttheta_h-\nu\boldsymbol{\Delta}^*\ttheta_h,\chi\rangle_{\Phi_\dive}
=&
\langle-\partial_t\rrho_h+\nu\rrho_h+\cP \ff-\cP_{\Phi}\ff,\chi\rangle_{\Phi_\dive}\\
&-\langle \cB(\uu,\uu)-\cI_{\Phi_\dive}\cB(\uu_h,\uu_h),\chi\rangle_{\Phi_\dive}\\
&-\langle \cC(\uu)-\cI_{\Phi_\dive}\cC(\uu_h),\chi\rangle_{\Phi_\dive}
\end{split}
\end{equation*}
holds almost everywhere on $[0,T_h)$.
\end{lemma} 
\begin{proof}
Using the property \eqref{2::equation::Ritz::1} of the discretised Ritz projection, we note that 
\begin{equation*}
-\boldsymbol{\Delta}^*\cR_{\Phi_\dive}\uu=-\boldsymbol{\Delta}^*\uu+\uu-\cR_{\Phi_\dive}\uu=-\boldsymbol{\Delta}^*\uu+\rrho_h
\end{equation*}
on $\cX$. Using \eqref{equation::NS_2::1}, we can rewrite the left hand to
\begin{equation*}
\begin{split}
\partial_t\ttheta_h-\nu\boldsymbol{\Delta}^*\ttheta_h
&=
\!\partial_t\cR_{\Phi_\dive}\uu-\nu\boldsymbol{\Delta}^*\cR_{\Phi_\dive}\uu-\cP_{\Phi}\left(\ff-\cB(\uu_h,\uu_h)-\cC(\uu_h)\right)\\
&=
\!\partial_t\cR_{\Phi_\dive}\uu-\nu\boldsymbol{\Delta}^*\uu+\nu\rrho_h-\cP_{\Phi}\left(\ff-\cB(\uu_h,\uu_h)-\cC(\uu_h)\right).
\end{split}
\end{equation*}
Solving \eqref{equation::NS_divfree::1} for $-\nu\boldsymbol{\Delta}^*\uu$, and inserting it into the equality above yields
\begin{equation*}
\begin{split}
\partial_t\ttheta_h-\nu\Delta^*\ttheta_h=&-\partial_t\rrho_h+\nu\rrho_h+\cP \ff-\cP_{\Phi}\ff\\
&-\cP(\cB(\uu,\uu))+\cP_{\Phi}(\cB(\uu_h,\uu_h))-\cP(\cC(\uu))+\cP_{\Phi}(\cC(\uu_h))
\end{split}
\end{equation*}
pointwise on $\cX$. Hence, for every $\chi\in \cF_{\Phi_\dive}(\cX)$ we have
\begin{equation*}
\begin{split}
\langle \partial_t\ttheta_h-\nu\boldsymbol{\Delta}^*\uu,\chi\rangle_{\Phi_\dive}=&\langle-\partial_t\rrho_h+\nu\rrho_h+\cP \ff-\cP_{\Phi}\ff,\chi\rangle_{\Phi_\dive}\\
&-\langle\cP(\cB(\uu,\uu))-\cP_{\Phi}(\cB(\uu_h,\uu_h)),\chi\rangle_{\Phi_\dive}\\
&-\langle\cP(\cC(\uu))-\cP_{\Phi}(\cC(\uu_h)),\chi\rangle_{\Phi_\dive}.
\end{split}
\end{equation*}
The fact that we have $\langle \cP \gggg,\chi\rangle_{\Phi_\dive}=\langle \gggg,\chi\rangle_{\Phi_\dive}$ for all $\gggg\in \HH^\sigma$ and $\chi\in \cF_{\Phi_\dive}(\cX)$ together with $\cP_{\Phi}=\cP \cI_{\Phi}=\cI_{\Phi_\dive}$ completes the proof.
\end{proof}

As we have now proved the equality in \eqref{Semi::eq::procedure}, we can go on in finding estimates for the terms on the right hand side. 
We will start with the second one, which addresses the convective part. Here, we will start with finding an estimate under the condition that the approximated solution is sufficiently good.
\begin{lemma}\label{Semi::lemma::2}
Let $\sigma>1$, $t_0\in(0,T)$ and $\uu(\cdot,t_0)\in \HH^{2\sigma+1}$. Let $M\geq1$ be a positive number such that $\|\uu(\cdot,t_0)-\uu_h(\cdot,t_0)\|_{\HH^\sigma_\dive}\leq M$. Then, for $t=t_0$, there exists a $C=C(\uu)>0$, independent of $\cX$ and $M$ but linear in $\|\uu\|_{C([0,T];H^{2\sigma}_\dive)}$, such that
\begin{equation*}
\begin{split}
\big\vert\langle \cB(\uu,\uu)-\cI_{\Phi_\dive}&\cB(\uu_h,\uu_h),\ttheta_h\rangle_{\Phi_\dive}\big\vert\\
\leq&
C h_{\cX}^\sigma M \left(1+\|\uu\|_{H^{2\sigma+1}_\dive}\right) \left(\|\ttheta_h\|_{\Phi_\dive}
+
\|\ttheta_h\|_{\HH^{\sigma+1}_\dive} \right)\\
&+
C\|\ttheta_h\|_{\Phi_\dive}^2
+
CM
\|\ttheta_h\|_{\HH^{\sigma+1}_\dive}\|\ttheta_h\|_{\Phi_\dive}.
\end{split}
\end{equation*} 
\end{lemma}
\begin{proof}
Firstly, we will split up the left hand side into
\begin{equation*}
\begin{split}
\langle \cB(\uu,\uu)-\cI_{\Phi_\dive}\cB(\uu_h,\uu_h),\ttheta_h\rangle_{\Phi_\dive}
=&
\langle \cB(\uu,\uu)-\cI_{\Phi_\dive}\cB(\uu,\uu),\ttheta_h\rangle_{\Phi_\dive}\\
&+
\langle \cI_{\Phi_\dive}\cB(\uu-\uu_h,\uu),\ttheta_h\rangle_{\Phi_\dive}\\
&+
\langle \cI_{\Phi_\dive}\cB(\uu_h,\uu-\uu_h),\ttheta_h\rangle_{\Phi_\dive}.
\end{split}
\end{equation*}
We will estimate the three terms separately. With Theorem \ref{2::Theorem::1} and Lemma \ref{auxiliary_results::lemma::1}, the first term can be estimated by
\begin{equation*}
\begin{split}
\vert\langle \cB(\uu,\uu)-\cI_{\Phi_\dive}\cB(\uu,\uu),\ttheta_h\rangle_{\Phi_\dive}\vert
&\leq 
\|\cB(\uu,\uu)-\cI_{\Phi_\dive}\cB(\uu,\uu)\|_{\Phi_\dive}\|\ttheta_h\|_{\Phi_\dive}\\
&\leq
Ch^\sigma_\cX\|\cB(\uu,\uu)\|_{H^{2\sigma}_\dive}\|\ttheta_h\|_{\Phi_\dive}\\
&\leq 
Ch^\sigma_\cX\|\uu\|_{\HH^{2\sigma}_\dive}\|\uu\|_{\HH^{2\sigma+1}_\dive}\|\ttheta_h\|_{\Phi_\dive}.
\end{split}
\end{equation*}
Using the boundedness of the interpolation operator $\|\cI_{\Phi_\dive} \vv\|_{\Phi_\dive}\leq\|\vv\|_{\Phi_\dive}$ for all $\vv\in \HH^{\sigma}_\dive$, we have for the second term
\begin{equation*}
\begin{split}
\vert\langle \cI_{\Phi_\dive}\cB(\uu-\uu_h,\uu),\ttheta_h\rangle_{\Phi_\dive}\vert
&\leq
\|\cB(\uu-\uu_h,\uu)\|_{\Phi_\dive}\|\ttheta_h\|_{\Phi_\dive}\\
&\leq
C\|\uu-\uu_h\|_{\HH^\sigma_\dive}\| \uu\|_{\HH^{\sigma+1}_\dive}\|\ttheta_h\|_{\Phi_\dive},
\end{split}
\end{equation*}
where, using the decomposition $\uu-\uu_h=\rrho_h+\ttheta_h$, the first norm can be estimated with Theorem \ref{2::Theorem::1} by
\begin{equation}\label{4::eq::proof::2}
\|\uu-\uu_h\|_{\HH^\sigma_\dive}\leq \|\rrho_h\|_{\HH^\sigma_\dive}+\|\ttheta_h\|_{\HH^\sigma_\dive}\leq Ch^\sigma_\cX\|\uu\|_{\HH^{2\sigma}_\dive}+\|\ttheta_h\|_{\HH^\sigma_\dive}.
\end{equation}
For the third term we need a little more effort. We begin by splitting it up by
\begin{equation}\label{4::eq::proof::1}
\begin{split}
\langle \cI_{\Phi_\dive}\cB(\uu_h,\uu-\uu_h),\ttheta_h\rangle_{\Phi_\dive}
=&
\langle \cI_{\Phi_\dive}\cB(\uu_h,\rrho_h),\ttheta_h\rangle_{\Phi_\dive}\\
&+
\langle \cI_{\Phi_\dive}\cB(\uu_h,\ttheta_h),\ttheta_h\rangle_{\Phi_\dive}.
\end{split}
\end{equation}
Using the self-adjointness of the interpolation operator $\cI_{\Phi_\dive}$ to split up the first part yields
\begin{equation*}
\begin{split}
\langle \cI_{\Phi_\dive}\cB(\uu_h,\rrho_h),\ttheta_h\rangle_{\Phi_\dive}
=&
\langle \cB(\uu_h,\rrho_h),\cI_{\Phi_\dive}\ttheta_h\rangle_{\Phi_\dive}\\
=&\langle \cB(\uu_h,\rrho_h),\ttheta_h\rangle_{\Phi_\dive}\\
&+
\langle \cB(\uu_h,\rrho_h),\cI_{\Phi_\dive}\ttheta_h-\ttheta_h\rangle_{\Phi_\dive},
\end{split}
\end{equation*}
where we have on the one hand the estimate
\begin{equation*}
\begin{split}
\vert\langle \cB(\uu_h,\rrho_h),\ttheta_h\rangle_{\Phi_\dive}\vert
&\leq
C\|\uu_h\|_{\HH^{\sigma-1}_\dive}\|\rrho_h\|_{\HH^{\sigma}_\dive}\|\ttheta_h\|_{\HH^{\sigma+1}_\dive}\\
&\leq
C h_{\cX}^\sigma(\|\uu\|_{\HH^{\sigma-1}_\dive}+M)\|\uu\|_{\HH^{2\sigma}_\dive}\|\ttheta_h\|_{\HH^{\sigma+1}_\dive}
\end{split}
\end{equation*}
and on the other hand
\begin{equation*}
\begin{split}
\vert\langle \cB(\uu_h,\rrho_h),\cI_{\Phi_\dive}\ttheta_h-\ttheta_h\rangle_{\Phi_\dive}\vert
&\leq
C\|\uu_h\|_{\HH^{\sigma}_\dive}\|\rrho_h\|_{\HH^{\sigma+1}_\dive}\|\cI_{\Phi_\dive}\ttheta_h\!-\!\ttheta_h\|_{\HH^{\sigma}_\dive}\\
&\leq
C(\|\uu\|_{\HH^{\sigma}_\dive}\!\!+\!M)h_{\cX}^{\sigma-1}\|\uu\|_{\HH^{2\sigma}_\dive}h_{\cX}\|\ttheta_h\|_{\HH^{\sigma+1}_\dive}\\
&\leq
Ch_{\cX}^{\sigma}(\|\uu\|_{\HH^{\sigma}_\dive}+M)\|\uu\|_{\HH^{2\sigma}_\dive}\|\ttheta_h\|_{\HH^{\sigma+1}_\dive}.
\end{split}
\end{equation*}
Finally, using the Cauchy-Schwarz inequality once again, the second part of \eqref{4::eq::proof::1} can be bounded by
\begin{equation*}
\begin{split}
\vert\langle \cI_{\Phi_\dive}\cB(\uu_h,\ttheta_h),\ttheta_h\rangle_{\Phi_\dive}\vert
&\leq 
C\|\uu_h\|_{\HH^{\sigma}_\dive}\|\ttheta_h\|_{\HH^{\sigma+1}_\dive}\|\ttheta_h\|_{\Phi_\dive}\\
&\leq 
C(\|\uu\|_{\HH^{\sigma}_\dive}+M)\|\ttheta_h\|_{\HH^{\sigma+1}_\dive}\|\ttheta_h\|_{\Phi_\dive}.
\end{split}
\end{equation*}
Summing up all the inequalities from above finishes the proof.
\end{proof}
Now we will give an estimate for the third term of the left hand side of \eqref{Semi::eq::procedure}.
\begin{lemma}\label{Semi::lemma::3}
Let $\sigma>1$, $t\in(0,T)$ and $\uu(\cdot,t)\in \HH^{\sigma+1}$. There exists a constant $C>0$ independent of $\uu$, $\cX$ and $t$, such that
\begin{equation*}
\big\vert\langle \cC(\uu)-\cI_{\Phi_\dive}\cC(\uu_h),\ttheta_h\rangle_{\Phi_\dive}\big\vert
\leq
C  h_{\cX}^\sigma \|\uu\|_{\HH^{2\sigma}_\dive}\|\ttheta_h\|_{\Phi_\dive} +
C\|\ttheta_h\|_{\Phi_\dive}^2.
\end{equation*} 
\end{lemma}
\begin{proof}
First, we split up the left hand side into
\begin{equation*}
\begin{split}
\langle \cC(\uu)-\cI_{\Phi_\dive}\cC(\uu_h),\ttheta_h\rangle_{\Phi_\dive}
=&\langle \cC(\uu)-\cI_{\Phi_\dive}\cC(\uu),\ttheta_h\rangle_{\Phi_\dive}\\
&+\langle \cI_{\Phi_\dive}\left(\cC(\uu)-\cC(\uu_h)\right),\ttheta_h\rangle_{\Phi_\dive}.
\end{split}
\end{equation*}
Using the Cauchy-Schwarz inequality, Theorem \ref{2::Theorem::1} and the fact that the operator $\cC$ is bounded and linear, the first term can easily be estimated by
\begin{equation*}
\begin{split}
\langle \cC(\uu)-\cI_{\Phi_\dive}\cC(\uu),\ttheta_h\rangle_{\Phi_\dive}
&\leq \|\cC(\uu)-\cI_{\Phi_\dive}\cC(\uu)\|_{\HH^\sigma_\dive}\|\ttheta_h\|_{\Phi_\dive}\\
&\leq Ch^\sigma_\cX\|\cC(\uu)\|_{\HH^{2\sigma}_\dive}\|\ttheta_h\|_{\Phi_\dive}\\
&\leq Ch^\sigma_\cX\|\uu\|_{\HH^{2\sigma}_\dive}\|\ttheta_h\|_{\Phi_\dive}.
\end{split}
\end{equation*}
Again, Cauchy-Schwarz and the boundedness of the operator $\cC$ yield
\begin{equation*}
\begin{split}
\langle \cI_{\Phi_\dive}\left(\cC(\uu)-\cC(\uu_h)\right),\ttheta_h\rangle_{\Phi_\dive}
\leq & C \|\uu-\uu_h\|_{\HH^\sigma_\dive}\|\ttheta_h\|_{\Phi_\dive},
\end{split}
\end{equation*}
where the first norm on the right hand side can be estimated by \eqref{4::eq::proof::2}.
\end{proof}
We will now prove a local error estimate for our semi-discrete solution. Note that we will write $\uu(t)=\uu(\cdot,t)$ for the sake of readability.

\begin{lemma}\label{Convergence::lemma::1}
Let $\sigma>1$, $\ff\in L^\infty(0,T_h;\HH^{\sigma})$ and $\uu(t)\in C^1([0,T_h),\HH^{2\sigma+1})$. Let $M\geq1$ such that
\begin{equation*}
T_M:=\sup\{t\in [0,T_h)\;\vert\;\max_{s\in[0,t]}\|\uu(\cdot,s)-\uu_h(\cdot,s)\|_{\HH^\sigma_\dive}\leq M\}
\end{equation*}
exists. Then, there exist two constants $C_1=C_1(\uu,\ff,\nu),C_2=C_2(\uu,\ff,\nu)>0$, independent of $\cX$ and $M$, such that 
%\begin{equation*}
%\sup_{t\in[0,T_M)}\|\uu(t)-\uu_h(t)\|+\nu h_\cX\|\uu-\uu_h\|_{L^2(0,T_M;\HH^{\sigma+1}_\dive)}\leq C_1M^2e^{C_2M^2T_M}h_\cX^\sigma.
%\end{equation*}
\begin{equation*}
\sup_{t\in[0,T_M)}\!\left(\|\uu(t)-\uu_h(t)\|_{\HH^{\sigma}_\dive}\!\!+\!\nu h_\cX\|\uu(t)-\uu_h(t)\|_{\HH^{\sigma+1}_\dive}\right)\!\leq\! C_1M^2e^{C_2M^2T_M}h_\cX^\sigma.
\end{equation*}
\end{lemma}
\begin{proof}
Using Lemma \ref{Semi::lemma::1} with $\chi=\ttheta_h$ yields
\begin{equation*}
\begin{split}
\frac{1}{2}\partial_t\|\ttheta_h\|^2_{\Phi_\dive}
-\nu\langle\DDelta^*\ttheta_h,\ttheta_h\rangle_{\Phi_\dive}
=&
\left\langle -\partial_t\rrho_h+\nu\rrho_h+\ff-\cI_{\Phi_\dive}f, \ttheta_h\right\rangle_{\Phi_\dive}\\
&-\left\langle \cB(\uu,\uu)-\cI_{\Phi_\dive}\cB(\uu_h,\uu_h) ,\ttheta_h\right\rangle_{\Phi_\dive}\\
&-\left\langle \cC(\uu)-\cI_{\Phi_\dive}\cC(\uu_h),\ttheta_h\right\rangle_{\Phi_\dive},
\end{split}
\end{equation*}
where the first term can be bounded by
\begin{equation*}
\begin{split}
\langle -\partial_t\rrho_h+\nu\rrho_h&+\ff-\cI_{\Phi_\dive}\ff, \ttheta_h\rangle_{\Phi_\dive}
\\
\leq&\left(\|\partial_t\rrho_h\|_{\Phi_\dive}+\nu\|\rrho_h\|_{\Phi_\dive}+\|\ff-\cI_{\Phi_\dive}\ff\|_{\Phi_\dive}\right)\|\ttheta_h\|_{\Phi_\dive}.
\end{split}
\end{equation*}
The first two norms can be estimated by Theorem \ref{approximation::theroem::2}, while the third one can be estimated by Theorem \ref{2::Theorem::1}. Together with Lemma \ref{Semi::lemma::2} and Lemma \ref{Semi::lemma::3} we can derive the estimate
\begin{equation*}
\begin{split}
\frac{1}{2}\partial_t\|\ttheta_h\|^2_{\Phi_\dive}
-
\nu\langle\DDelta^*&\ttheta_h,\ttheta_h\rangle_{\Phi_\dive}\\
\leq&
Ch_\cX^\sigma\left(\|\partial_t\uu\|_{\HH^{2\sigma}_\dive}
+\|\uu\|_{\HH^{2\sigma}_\dive}
+\|\ff\|_{\HH^{2\sigma}_\dive}\right)\|\ttheta_h\|_{\Phi_\dive}\\
&+
C h_{\cX}^\sigma M \left(1+\|\uu\|_{H^{2\sigma+1}_\dive}\right) \left(\|\ttheta_h\|_{\Phi_\dive}
+
\|\ttheta_h\|_{\HH^{\sigma+1}_\dive} \right)\\
&+
C\|\ttheta_h\|_{\Phi_\dive}^2
+
CM\|\ttheta_h\|_{\Phi_\dive}\|\ttheta_h\|_{\HH^{\sigma+1}_\dive}\\
&
+C h_{\cX}^\sigma \|\uu\|_{H^{2\sigma}_\dive}\|\ttheta_h\|_{\Phi_\dive} 
+C\|\ttheta_h\|_{\Phi_\dive}^2
\end{split}
\end{equation*}
We will use the Cauchy-Schwarz inequality to separate the terms including $\ttheta_h$. 
Especially, using \eqref{4::eq::equivalece_inner_product}, we will choose a constant $c_\nu>0$, depending on $\nu$, such that
\begin{equation*}
c_\nu\|\ttheta_h\|_{\HH^{\sigma+1}_\dive}^2\!<\!-\frac{\nu}{2}\langle\DDelta^*\ttheta_h,\!\ttheta_h\rangle_{\Phi_\dive}.
\end{equation*}
Collecting all terms including $\|\ttheta_h\|_{\HH^{\sigma+1}_\dive}$ and using the Cauchy-Schwarz inequality we can derive
\begin{equation*}
\begin{split}
C h_{\cX}^\sigma M &\left(1+\|\uu\|_{H^{2\sigma+1}_\dive}\right)\|\ttheta_h\|_{\HH^{\sigma+1}_\dive}
+CM\|\ttheta_h\|_{\Phi_\dive}\|\ttheta_h\|_{\HH^{\sigma+1}_\dive}\\
\leq&
2C^2 h_{\cX}^{2\sigma} M^2 c_\nu^{-1}\left(1+\|\uu\|_{H^{2\sigma+1}_\dive}\right)^2+C^2M^2c_\nu^{-1}\|\ttheta_h\|_{\Phi_\dive}^2 +c_\nu\|\ttheta_h\|_{\HH^{\sigma+1}_\dive}^2
\end{split}
\end{equation*}
Overall, we arrive at the estimate
\begin{equation*}
\begin{split}
\frac{1}{2}\partial_t\|\ttheta_h\|^2_{\Phi_\dive}
\!\!-\!
\nu\langle\DDelta^*\ttheta_h,\ttheta_h\rangle_{\Phi_\dive}
\!
\leq&
CM^2h_\cX^{2\sigma}\!\left(
\|\partial_t\uu\|_{\HH^{2\sigma}_\dive}^2
+\|\uu\|_{\HH^{2\sigma}_\dive}^2\right)\\
&+
CM^2h_\cX^{2\sigma}\!\left(
\|\uu\|_{\HH^{2\sigma+1}_\dive}^2
+\|\ff\|_{\HH^{2\sigma}_\dive}^2\right)\\
&
+C(1+M)\|\ttheta_h\|_{\Phi_\dive}^2
+c_\nu\|\ttheta_h\|_{\HH^{\sigma+1}_\dive},
\end{split}
\end{equation*}
Rearranging this inequality yields
\begin{equation}\label{Convergence::proof::lemma::1::eq::1}
\begin{split}
\frac{1}{2}\partial_t\|\ttheta_h\|^2_{\Phi_\dive}
-\frac{\nu}{2}\langle\DDelta^*\ttheta_h,\ttheta_h\rangle_{\Phi_\dive}
\leq&
CM^2h_\cX^{2\sigma}\left(
\|\partial_t\uu\|_{\HH^{2\sigma}_\dive}^2
+\|\uu\|_{\HH^{2\sigma}_\dive}^2\right)\\
&
+CM^2h_\cX^{2\sigma}\left(
\|\uu\|_{\HH^{2\sigma+1}_\dive}^2
+\|\ff\|_{\HH^{2\sigma}_\dive}^2\right)\\
&
+C(1+M)\|\ttheta_h\|_{\Phi_\dive}^2.
\end{split}
\end{equation}
Hence, for every $t\in[0,T_M)$, Gronwall's inequality implies
\begin{equation*}
\begin{split}
\|\ttheta_h(t)\|^2_{\Phi_\dive}
\leq&
\|\ttheta_h(0)\|^2_{\Phi_\dive}
+CM^2e^{CM^2t}h_\cX^{2\sigma}\|\uu\|_{H^1(0,T_M;\HH^{2\sigma}_\dive)}^2\\
&
+CM^2e^{CM^2t}h_\cX^{2\sigma}\left(\|\uu\|_{L^2(0,T_M;\HH^{2\sigma+1}_\dive)}^2
+\|\ff\|_{L^2(0,T_M;\HH^{2\sigma}_\dive)}^2\right).
\end{split}
\end{equation*}
The initial error can be bounded by
\begin{equation*}
\|\ttheta_h(0)\|^2_{\Phi_\dive}\leq
 \| \cR_{\Phi_\dive} \uu_0-\uu_0 \|^2_{\Phi_\dive}
+\| \uu_0 - \cI_{\Phi_\dive}\uu_0 \|^2_{\Phi_\dive}\leq Ch_\cX^\sigma\|\uu_0\|_{\HH^{2\sigma}_\dive},
\end{equation*}
which finishes the proof for the first part.
For the second part we can use \eqref{4::eq::equivalece_inner_product} again and insert the just derived estimation in \eqref{Convergence::proof::lemma::1::eq::1} and
integrate with respect to the time variable. This yields the same error bound as above for $\|\ttheta_h\|_{L^2(0,T_M;\HH^{\sigma+1}_\dive)}$. For the projection error $\rrho_h$, the approximation order is unfortunately reduced by one because the error is measured in the $\HH^{\sigma+1}_\dive$ norm.
\end{proof}
Now, we will prove our main result of this paper. As the local in time solution has been given in Lemma \ref{Convergence::lemma::1}, it remains to show that the existence of the semidiscrete solution only depends on the existence interval of the solution of the projected Navier-Stokes equations. 
If the conditions of Lemma \ref{3::Lemma::existence} are satisfied and this solution exists global-in-time, our approximation will exist global in time as well.
\begin{theorem}\label{4::theroem::main_theorem}
Let $\sigma>3$. Suppose the initial velocity satisfies $\uu_0\in\HH^{2\sigma+1}_\dive$ and the right hand side satisfies $\ff\in C([0,\infty);\HH^{2\sigma})$. Let $\uu$ be the global solution of the projected Navier-Stokes equations \eqref{equation::NS_divfree::1} - \eqref{equation::NS_divfree::2} on a maximum time interval $[0,T]$ for a $T<\infty$ or on $[0,\infty)$.
Moreover, let $\cX\subset\SSS^d$ be a finite set of pair-wise distinct points with fill distance $h_\cX\leq h_0$ for some $h_0>0$.

Then, the semi-discrete solution $\uu_h\in C^1([0,T_h);\cF_{\Phi_\dive}(\cX))$ of the discretised Navier-Stokes equations exists in all of $[0,T]$ or $[0,\infty)$, respectively, and there exist constants $C_1,C_2>0$, depending on $\ff$, $\nu$, $\uu_0$ and $\uu$, but independent of $\cX$, such that
\begin{equation*}
\|\uu(t)-\uu_h(t)\|_{\HH^\sigma_\dive}+\nu h_\cX\|\uu(t)-\uu_h(t)\|_{\HH^{\sigma+1}_\dive}<C_1e^{C_2t}h_\cX^\sigma.
\end{equation*}
for all $t\in[0,T]$ or $t\in[0,\infty)$, respectively.
\end{theorem}
\begin{proof}
Since the initial error is just the interpolation error of $\uu_0$, it can be bounded by $C h_\cX^\sigma$. 
Now choose $h_1$ small enough such that $\|\uu(0)-\uu_h(0)\|_{\HH^\sigma_\dive}< 1$ for all $\cX\subset\SSS^2$  with $h_\cX<h_1$ and fix $M:=2$. 
Since $\uu$ and $\uu_h$ are continuous in time, we have that
\begin{equation*}
T_M=\sup\left\{t\in [0,T_h)\;\vert\;\max_{s\in[0,t]}\|\uu(s)-\uu_h(s)\|_{\HH^\sigma_\dive}\leq 2\right\}
\end{equation*}
exists with $T_M>0$. Applying Lemma \ref{Convergence::lemma::1} yields that there exists a constant $C>0$ such that
\begin{equation*}
\|\uu(t)-\uu_h(t)\|_{\HH^\sigma_\dive}\leq C h_\cX^\sigma
\end{equation*}
for all $t\in[0,T_M)$. Hence, we find a $0<h_0<h_1$ such that $\|\uu(t)-\uu_h(t)\|_{\HH^\sigma_\dive}\leq 1$ for all $t\in[0,T_M)$ and $h_\cX<h_1$.

To complete the proof, it remains to show that $T_M=T$ or $T_M=\infty$, respectively. 
We will concentrate on the second case. 
The proof for $T_M=T$ works analogously.

First, suppose that $T_M<T_h$.
Since $\uu$ and $\uu_h$ are continuous in time, we know that 
$
\|\uu(T_M)-\uu_h(T_M)\|_{\HH^\sigma_\dive}\leq 1$.
Moreover, there exists an $\epsilon>0$ such that $T_M+\epsilon\in [0,T_h)$ and
\begin{equation*}
\|\uu(T_M)-\uu_h(T_M)-(\uu(T_M+t)-\uu_h(T_M+t))\|_{\HH^\sigma_\dive}\leq 1
\end{equation*}
for all $t\in[0,\epsilon)$. But then we have $
\|\uu(t)-\uu_h(t)\|_{\HH^\sigma_\dive}\leq 2$ for all $t\in[0,T_M+\veps)$, which is a contradiction to the definition of $T_M$. 
Hence, $T_M\geq T_h$.

Moreover, $\uu_h$ is bounded on $ \SSS^2\times[0,T_h)$ since we have
\begin{equation*}
\vert u_h(\xx,t)\vert\leq
\|\uu_h(t)\|_{\HH^\sigma_\dive}
\!\leq
\|\uu(t)\|_{\HH^\sigma_\dive}
\!+
\|\uu(t)-\uu_h(t)\|_{\HH^\sigma_\dive}
\!\leq
\|\uu(t)\|_{\HH^\sigma_\dive}+M
\end{equation*}
for all $\xx\in\SSS^2$ and $t\in[0,T_h)$. But according to the continuity criterion for ODEs, this would mean that our numerical solution would exists globally in time. Hence, we have $T_M\geq T_h=\infty$.
\end{proof}

\subsection{Calculation of the pressure}\label{4::pressure}
To obtain an approximation for the pressure $p$ at least up to a constant, we have to solve the curl-free projection of the Navier-Stokes equations \eqref{equation::NS_curlfree}. The approximated version of this equation is given by 
\begin{equation}\label{equation::NS_curlfree::projected}
\nabla_* p_h=(\operatorname{id}-\cP)\cI_{\Phi} \left(\ff-\cB(\uu_h,\uu_h)-\cC(\uu_h)\right),\quad\text{in }\SSS^{2}\times(0,T).
\end{equation} 
Fortunately, we can obtain this approximation without much effort.
Since we already need to calculate the right hand side of \eqref{equation::NS_2::1}, we already have time-dependent coefficients $\bbeta_j(t)=\bbeta_j\in T_{\xx_j}(\SSS^2)$, $1\leq j\leq N$, such that
\begin{equation*}
\cP\cI_{\Phi} \left(\ff-\cB(\uu_h,\uu_h)-\cC(\uu_h)\right)
=
\sum_{j=1}^N\Phi_\dive(\cdot,\xx_j)\bbeta_j.
\end{equation*}
Just using these coefficients, we can give an approximation to the pressure by simply exchanging the kernel function by its curl-free part and using \eqref{2::eq::Phi_curl_long} to arrive at
\begin{equation*}
\begin{split}
(\operatorname{id}-\cP)\cI_{\Phi} \left(\ff-\cB(\uu_h,\uu_h)-\cC(\uu_h)\right)
&=
\sum_{j=1}^N\Phi_\curl(\cdot,\xx_j)\bbeta_j\\
&=
\nabla^*\sum_{j=1}^N(\nabla_{}^{*})^T\phi(\cdot,\xx_j)\bbeta_j.
\end{split}
\end{equation*}
Hence an approximation of the pressure $p$, up to a constant, is simply given by
\begin{equation*}
p_h=\sum_{j=1}^N(\nabla^*)^T\phi(\cdot,\xx_j)\bbeta_j.
\end{equation*}
Since the integral over the surface-gradient of $\phi(\cdot,\xx_j)\in H^\sigma$ vanishes for every $1\leq j\leq N$, see \cite{Freeden1998}, the mean integral of our approximation $p_h$ vanishes, too. Therefore, we assume that the integral over the true pressure p also vanishes.
The following error estimate for the approximation of the pressure can be established.

\begin{theorem}\label{4::theroem::main_theorem::pressure}
Under the assumptions of Theorem \ref{4::theroem::main_theorem} and that the meant integral of $p$ vanishes, there exists a constant $C>0$ such that
\begin{equation*}
\|p-p_h\|_{L^2(0,T;\!H^{\sigma+1})}\leq Ch_{\cX}^{\sigma-1}
\end{equation*}
holds for all data sites $\cX\in\SSS^2$ with fill distance $h_\cX\leq h_0$ and all $t\in[0,T)$.
\end{theorem}
\begin{proof}
Since both $p$ and $p_h$ have zero mean integral, we have
\begin{equation*}
\|p-p_h\|_{H^{\sigma+1}}\leq \sqrt{2}\|\nabla^*p-\nabla^*p_h\|_{\HH^\sigma_\curl}.
\end{equation*}
Inserting equation \eqref{equation::NS_curlfree} and \eqref{equation::NS_curlfree::projected} for $\nabla^*p$ and $\nabla^*p_h$, respectively, and splitting up the singe terms yields
\begin{equation*}
\begin{split}
\|p-p_h\|_{H^{\sigma+1}}
\leq&
\|\ff-\cI_{\Phi} \ff\|_{\HH^{\sigma}_\curl}
+\|\cB(\uu,\uu)-\cI_{\Phi} \cB(\uu_h,\uu_h)\|_{\HH^{\sigma}_\curl}\\
&+\|\cC(\uu)-\cI_{\Phi} \cC(\uu_h)\|_{\HH^{\sigma}_\curl}.
\end{split}
\end{equation*}
With Theorem \ref{2::Theorem::1}, the first term can be estimated by
\begin{equation*}
\|\ff-\cI_{\Phi} \ff\|_{\HH^{\sigma}_\curl}\leq ch^\sigma_{\cX}\|\ff\|_{\HH^{\sigma}_\curl}.
\end{equation*}
The second term can be splitted up to
\begin{equation*}
\begin{split}
\|\cB(\uu,\uu)-\cI_{\Phi} \cB(\uu_h,\uu_h)\|_{\HH^{\sigma}_\curl}
\leq&
\|\cB(\uu,\uu)-\cI_{\Phi} \cB(\uu,\uu)\|_{\HH^{\sigma}_\curl}\\
&+
\|\cI_{\Phi} \cB(\uu,\uu)-\cI_{\Phi} \cB(\uu_h,\uu_h)\|_{\HH^{\sigma}_\curl},
\end{split}
\end{equation*}
where the first part can simply be bounded by
\begin{equation*}
\|\cB(\uu,\uu)-\cI_{\Phi} \cB(\uu,\uu)\|_{\HH^{\sigma}_\curl}
\leq
ch_{\cX}^\sigma\|\cB(\uu,\uu)\|_{\HH^{2\sigma}_\curl}
\leq
ch_{\cX}^\sigma\|\uu\|_{\HH^{2\sigma}_\curl}\|\uu\|_{H^{2\sigma+1}_\curl}.
\end{equation*}
For the second part, we use that the operator $\cI_{\Phi}$ is bounded in $\HH^\sigma_\curl$. Hence, we have
\begin{equation*}
\begin{split}
\|\cI_{\Phi}(\cB(\uu,\uu)-\cB(\uu_h,\uu_h))\|_{\HH^{\sigma}_\curl}
\leq &
\|\cB(\uu,\uu)-\cB(\uu_h,\uu_h)\|_{\HH^{\sigma}_\curl} \\
\leq &
\|\cB(\uu-\uu_h,\uu)\|_{\HH^{\sigma}_\curl}\\
&+
\|\cB(\uu_h,\uu-\uu_h)\|_{\HH^{\sigma}_\curl}\\
\leq &
C\|\uu-\uu_h\|_{\HH^{\sigma+1}_\curl}.
\end{split}
\end{equation*}
The last term can be split and estimated like the previous one. We obtain
\begin{equation*}
\begin{split}
\|\cC(\uu)\!-\cI_{\Phi} \cC(\uu_h)\|_{\HH^{\sigma}_\curl}
\leq&
\|\cC(\uu)\!-\cI_{\Phi} \cC(\uu)\|_{\HH^{\sigma}_\curl}
\!\!+\!
\|\cI_{\Phi}\left(\cC(\uu-\uu_h)\right)\|_{\HH^{\sigma}_\curl}\\
\leq&
 C\left(h_{\cX}^\sigma\|\uu\|_{\HH^{2\sigma}_\curl}+\|\uu-\uu_h\|_{\HH^{\sigma}_\curl}\right).
\end{split}
\end{equation*}
Inserting the results from Theorem \ref{4::theroem::main_theorem}, squaring and integrating over time finishes the proof.
\end{proof}

\section{Numerical considerations}\label{5}
In this section we will rewrite the spatial-discretized Navier-Stokes equations \eqref{equation::NS_2::1} - \eqref{equation::NS_2::2} into a time-depended linear system, which can then be solved using standard ODE solvers. 
Moreover, we take a look at the numerical effort to show that the given method can be evaluated efficiently. 
In the end, we will evaluate the new method on a standard benchmark problem, see \cite{Debussche1993,fengler2004,Ganesh2010}, to show the stability of the new method.

\subsection{Algorithm and Cost}
Using the representation of the approximated solution \eqref{equation::approx::1}, we can write the discretised Navier-Stokes equations \eqref{equation::NS_2::1} - \eqref{equation::NS_2::2} in matrix-vector form given by
\begin{align}\label{Algo::eq::disc_NS}
A_{\Phi_\dive}\dot{\aalpha}
&=
\nu A_{\Delta^*\Phi_\dive}\aalpha
+
\gggg(\aalpha)
\quad \text{on }(0,T),\\\label{Algo::eq::disc_NS::2}
A_{\Phi_\dive}\aalpha(0)
&=(\uu_0(\xx_k))_{1\leq k\leq N},
\end{align}
where the kernel-based $3N\times 3N$-matrices are given by 
\begin{equation*}
A_{\Phi_\dive}=(\Phi_\dive(\xx_k,\xx_j))_{1\leq j,k\leq N},\quad A_{\Delta^*\Phi_\dive}=(\Delta^*\Phi_\dive(\xx_k,\xx_j))_{1\leq j,k\leq N}.
\end{equation*}
Note that both matrices $A_{\Phi_\dive}$ and $-A_{\Delta^*\Phi_\dive}$ are positive definite on a $2N$-dimensional subspace of $\R^{3N}$ since the kernel function $\Phi_\dive$ as well as $-\Delta^*\Phi_\dive$ are positive definite on the tangent space of the sphere.
To solve the $3N\times 3N$ linear systems \eqref{Algo::eq::disc_NS} and \eqref{Algo::eq::disc_NS::2} in the implementation, we can reduce it to a $2N \times 2N$ system by introducing coordinates and bases for the various tangent planes of the sphere. The corresponding matrices of the reduced system are classically positive definite. For details and the exact implementation, see \cite{Fuselier2009_2}.

The right hand side term $\gggg:\R^{3N}\to\R^{3N}$ is given by
\begin{equation}\label{Algo::eq::first::system}
\gggg(\aalpha):=\left(\cP_{\Phi}(\ff -\cB(\uu_h,\uu_h)-\cC(\uu_h))(\xx_k)\right)_{1\leq k\leq N},
\end{equation}
where $\uu_h$ is given by \eqref{equation::approx::1}. 
To calculate the value of $\gggg$ for a given $\aalpha$, we have to solve the interpolation problem
\begin{equation*}
\sum_{j=1}^N \Phi(\xx_k,\xx_j)\bbeta_j=(\ff -\cB(\uu_h,\uu_h)-\cC(\uu_h))(\xx_k),\quad 1\leq k\leq N.
\end{equation*}  
The function $\gggg$ is then given by $\gggg(\aalpha)=\sum_{j=1}^N \Phi_\dive(\cdot,\xx_j)\bbeta_j$.

%The convective term for the discretised functions $\cB(\vv,\ww)=[(\vv\cdot\nabla^*)\ww]$ is given by
%\begin{equation*}
%b_\cX(\aalpha,\widetilde{\aalpha}):= \cB\left( \sum_{j=1}^N \Phi_\dive(\cdot,\xx_j){\aalpha}_j,\sum_{j=1}^N\Phi_\dive(\cdot,\xx_j)\widetilde{\aalpha}_j\right)
%\end{equation*}
%and the discretised rotation term by
%\begin{equation*}
%c_\cX(\aalpha):=\cC(u_h)=l\nn\times \sum_{j=1}^N \Phi_\dive(\cdot,\xx_j){\aalpha}_j.
%\end{equation*}
%For the sake of readability, let
%\begin{equation*}
%g(\aalpha):=\left(\sum_{j=1}^N\Phi_\dive(\xx_k,\yy_j)\bbeta_j\right)_{1\leq k\leq N}=\left(\cP_{\Phi,X}(\ff -b_\cX(\aalpha,\aalpha)-c_\cX(\aalpha))(\xx_k)\right)_{1\leq k\leq N}
%\end{equation*}
%be an abbreviation for the term on right hand side with $\bbeta$ given as described in section \ref{4::pressure}.

To solve the received ODE, standard explicit or semi-implicit methods can be used. 
However, using explicit solvers, one has to be aware that the time step $\tau$ is probably coupled to the filling distance $h_\cX$ and has to satisfy some kind of CFL condition for the system to remain stable. however, for a investigation of the numerical cost of the new method, we will neglect the fact and restrict ourselves to two simple examples for the time discretization, the explicit and the semi-implicit Euler algorithm. Let $\tau>0$ be the time discretization parameter. Then, a time step for an explicit Euler scheme is given by 
\begin{equation}\label{Algo::eq::Euler::expl}
\aalpha_{n+1}=\aalpha_{n}+
\tau A_{\Phi_\dive}^{-1}
\left(
\nu A_{\Delta^*\Phi_\dive}\aalpha_n+\gggg(\aalpha_n)
\right)
\end{equation}
and a time step of the semi-implicit Euler scheme is given by
\begin{equation}\label{Algo::eq::Euler::impl}
\begin{split}
\aalpha_{n+1}&=
\tau\left(A_{\Phi_\dive}-\nu \tau A_{\Delta^*\Phi_\dive}\right)^{-1}
\left(
A_{\Phi_\dive}\aalpha_{n}+\gggg(\aalpha_n)
\right)\\
&=\aalpha_n+\tau\left(A_{\Phi_\dive}-\nu \tau A_{\Delta^*\Phi_\dive}\right)^{-1}
\left(
\nu A_{\Delta^*\Phi_\dive}\aalpha_n+\gggg(\aalpha_n).
\right)
\end{split}
\end{equation}
The costs of the given schemes depend highly on the ability to solve the two linear systems given in \eqref{Algo::eq::first::system} and in \eqref{Algo::eq::Euler::expl} or \eqref{Algo::eq::Euler::impl}, respectively. This would have a cost of $\mathcal{O}(N^3)$ for each linear system using standard solvers in each time-step. 
However these, costs can be reduced by various strategies.
Since the matrices are constant in time, one can take a Cholesky decomposition on the start of the cost of $\mathcal{O}(N^3)$. Then, the cost of the solution of the linear system in each time-step reduces to $\mathcal{O}(N^2)$.
Instead of a Cholesky decomposition, it is also possible to use Lagrange basis functions, with which it is theoretically possible to reduce the effort of one timestep to $\mathcal{O}(N^2)$. However, as in the case of Cholesky decomposition, it we need to calculate the Lagrange basis function at the beginning of the cost of $\mathcal{O}(N^3)$. 
We will give a little preview about the Lagrange basis in the conclusion. 
Note that both, the Cholesky decomposition as well as the Lagrange functions have to be calculated once for a given point set, which means that we can use the same decomposition for different initial conditions or external forces $\ff$.

Another possibility is to use the fact that the kernel functions have compact support. By using scaled kernel functions of the form $\phi_\epsilon=\phi(\cdot/\epsilon)$ for an $\epsilon>0$, the matrices appearing in the schemes above can be made sparse for small $\epsilon$. In this case, the effort can be reduced to $\mathcal{O}(N^2\log(N))$ by using special sparse solvers. However, it is to be expected that the theoretical convergence rate will depend on the scaling parameter $\epsilon$, which means that $\epsilon$ has to be coupled to the fill distance $h_\cX$ in a certain way in order to still have convergence for $h_\cX\to 0$ and $\epsilon\to 0$.

Of course, any other time discretization scheme would be possible. We refer to general IMEX (Implicit-Explicit) schemes, see \cite{Ascher1997}, which have better stability properties than explicit methods.
Furthermore, it is also possible to give fully implicit schemes, see \cite{Keim2016}. Unfortunately, those schemes are highly inefficient since in every time-step a non-linear equation system has to be solved. 

\subsection{Numerical Validation}

\begin{figure}
 \centering
 
\begin{subfigure}{.5\textwidth}
  \centering
  \includegraphics[width=\linewidth]{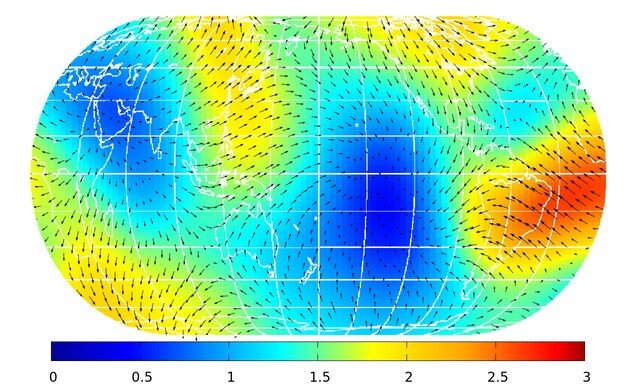}
  \caption*{Velocity field $\uu_h$ at time t=10}
\end{subfigure}%
\begin{subfigure}{.5\textwidth}
  \centering
  \includegraphics[width=\linewidth]{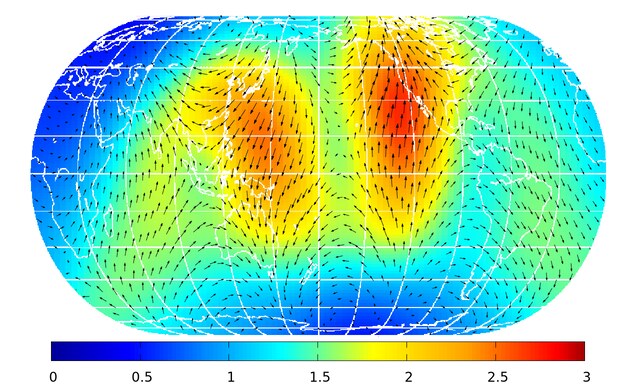}
  \caption*{Velocity field $\uu_h$ at time t=20}
\end{subfigure}

\begin{subfigure}{.5\textwidth}
  \centering
  \includegraphics[width=\linewidth]{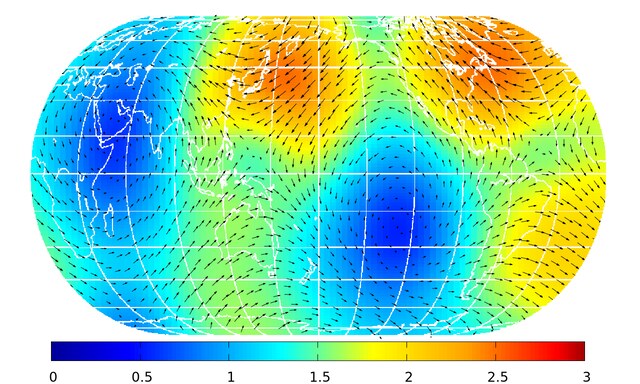}
  \caption*{Velocity field $\uu_h$ at time t=30}
\end{subfigure}%
\begin{subfigure}{.5\textwidth}
  \centering
  \includegraphics[width=\linewidth]{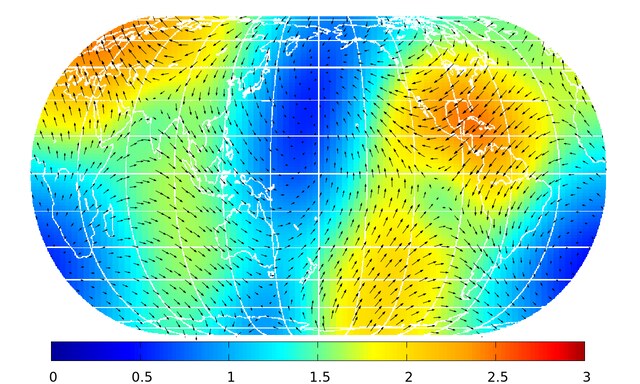}
  \caption*{Velocity field $\uu_h$ at time t=40}
\end{subfigure}

\begin{subfigure}{.5\textwidth}
  \centering
  \includegraphics[width=\linewidth]{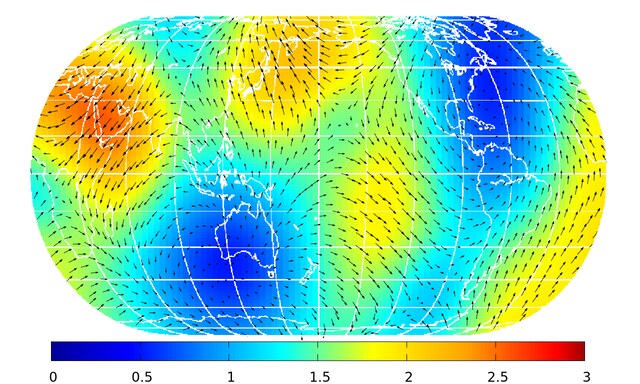}
  \caption*{Velocity field $\uu_h$ at time t=50}
\end{subfigure}%
\begin{subfigure}{.5\textwidth}
  \centering
  \includegraphics[width=\linewidth]{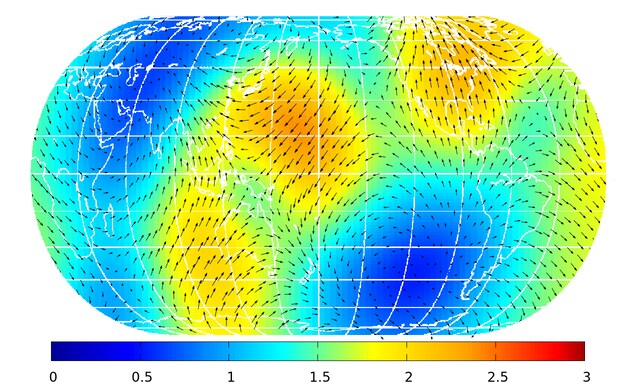}
  \caption*{Velocity field $\uu_h$ at time t=60}
\end{subfigure}
\caption{Velocity field $\uu_h$ at various times.}
\label{fig:velocity}
\end{figure}

We will now demonstrate the stability of the new method on a test case which is motivated by the numerical demonstration in \cite{fengler2004}, see also \cite{Debussche1993,Ganesh2010}. 
This simulation should only be seen as a numerical test to validate that the new method is stable. 
A proper numerical validation of the results in chapter 3 is unfortunately not possible, since the author is not aware of any analytical solution of the Navier-Stokes equations with which a numerical solution of the method could be compared.

\begin{figure}[htb]
 \centering
 
\begin{subfigure}{.5\textwidth}
  \centering
  \includegraphics[width=\linewidth]{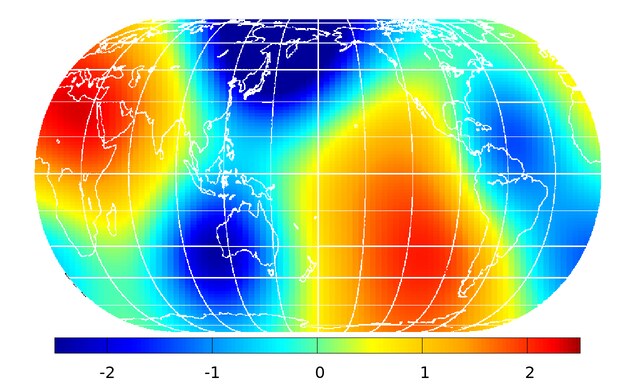}
  \caption*{Velocity field $\uu_h$ at time t=10}
\end{subfigure}%
\begin{subfigure}{.5\textwidth}
  \centering
  \includegraphics[width=\linewidth]{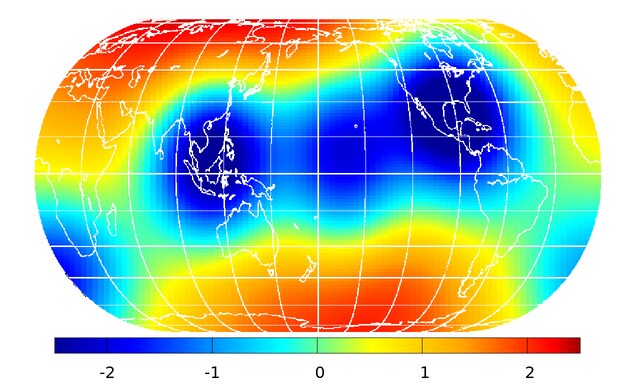}
  \caption*{Velocity field $\uu_h$ at time t=20}
\end{subfigure}

\begin{subfigure}{.5\textwidth}
  \centering
  \includegraphics[width=\linewidth]{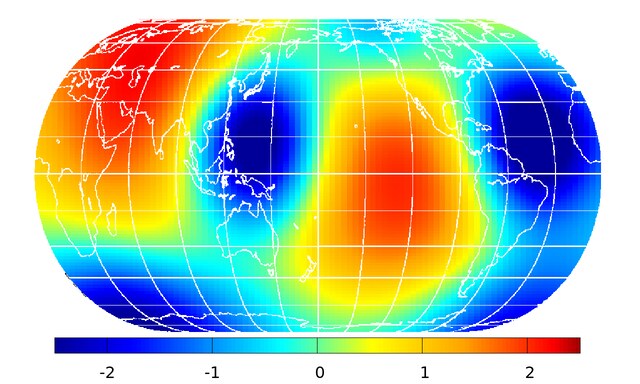}
  \caption*{Velocity field $\uu_h$ at time t=30}
\end{subfigure}%
\begin{subfigure}{.5\textwidth}
  \centering
  \includegraphics[width=\linewidth]{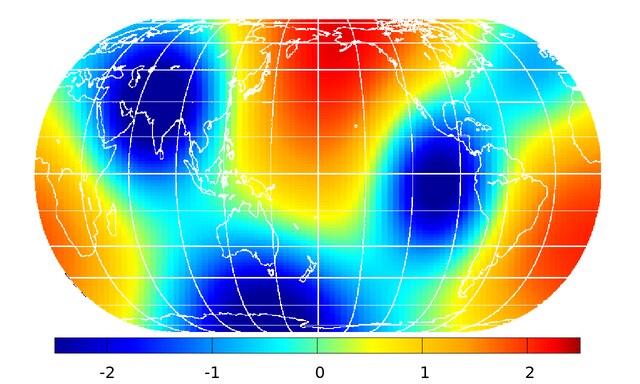}
  \caption*{Velocity field $\uu_h$ at time t=40}
\end{subfigure}

\begin{subfigure}{.5\textwidth}
  \centering
  \includegraphics[width=\linewidth]{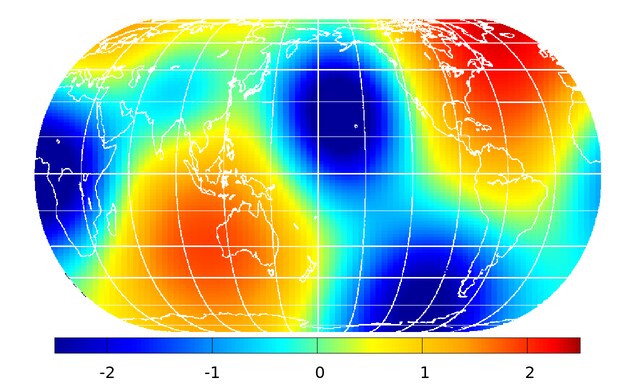}
  \caption*{Velocity field $\uu_h$ at time t=50}
\end{subfigure}%
\begin{subfigure}{.5\textwidth}
  \centering
  \includegraphics[width=\linewidth]{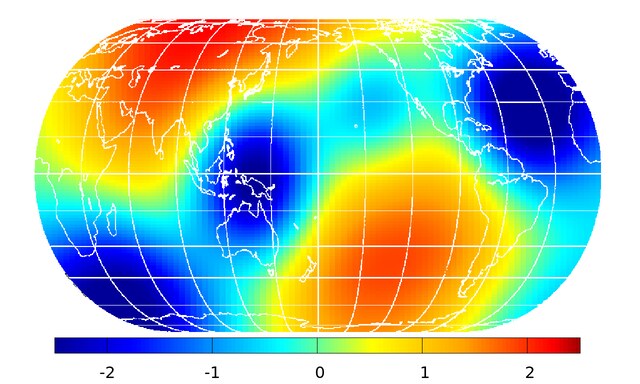}
  \caption*{Velocity field $\uu_h$ at time t=60}
\end{subfigure}
\caption{Pressure $p_h$ at various times.}
\label{fig:pressure}
\end{figure}

The initial velocity $\uu_0$, which is chosen in analogy to \cite{fengler2004}, is given by $\uu_0=\LL^* p$, where 
$p(\xx)=\sum_{\vert\nn\vert<20}a_{\nn}\xx^\nn$
is a polynomial of degree $19$ in $\R^3$ such that  $\|\uu_0\|_{\LL^2}=1$.
The coefficients $a_{\nn}$ were generated by a random function such that $a_{\nn}\sim\vert\nn\vert^{-1}$. This construction ensures that the initial velocity is a tangential and divergence-free velocity field. 
The time-depended external force $f$ is given by
\begin{equation*}
f(t,\xx)=\gamma(t)\yy_{3,0}(\xx)
\end{equation*}
with
\begin{equation*}
\gamma(t)=
\begin{cases}
1& \text{for } t\in[0,10],\\
\cos\left(\frac{\pi}{5}t\right)e^{(10-t)/5}& \text{for } t\in]10,60].
\end{cases}
\end{equation*}
Note that the initial condition as well as the forcing term both satisfy the conditions of Lemma \ref{3::Lemma::existence}, such that an analytical solution of the problem exists global in time.
The viscosity coefficient and the angular velocity are set to $\nu=10^{-4}$ and $\Omega=1$. 
The set of data points $\cX$ is given by $N=2500$ points which are distributed by minimizing the potential energy of the points, see \cite{Sloan2004}.
For the basis kernel function, we used the unscaled Wendland function $\phi_4$ from Table \ref{2::table::1}. 

As the time integration scheme, we used a third order implicit-explicit Runge-Kutta method which can be found in \cite[Section 2.4]{Ascher1997} with a timestep $\tau=10^{-2}$. We used the explicit solver in the IMEX scheme to handle the term $\gggg$ in \eqref{Algo::eq::disc_NS} and the implicit solver for the Laplace term.

The approximated velocity field $\uu_h$ as well as the approximated pressure $p_h$ can be seen in Figure \ref{fig:velocity} and Figure \ref{fig:pressure}, respectively, for $t = 10, 20, 30, 40, 50, 60$. 
The approximated $\LL^2$-norms of the velocity and the pressure are given by
\begin{equation*}
e_\uu(t)=\sqrt{\frac{4\pi}{N}\sum_{j=1}^N\|\uu_h(\xx_j,t)\|_2^2},\quad e_p(t)=\sqrt{\frac{4\pi}{N}\sum_{j=1}^N\vert p_h(\xx_j,t)\vert^2},\quad t\in[0,60].
\end{equation*}
The time evolution of both norms can be seen in Figure \ref{fig::time_evolution::energy}.
\begin{figure}[h]
\centering
\begin{subfigure}{.5\textwidth}
  \centering
  \includegraphics[width=\linewidth]{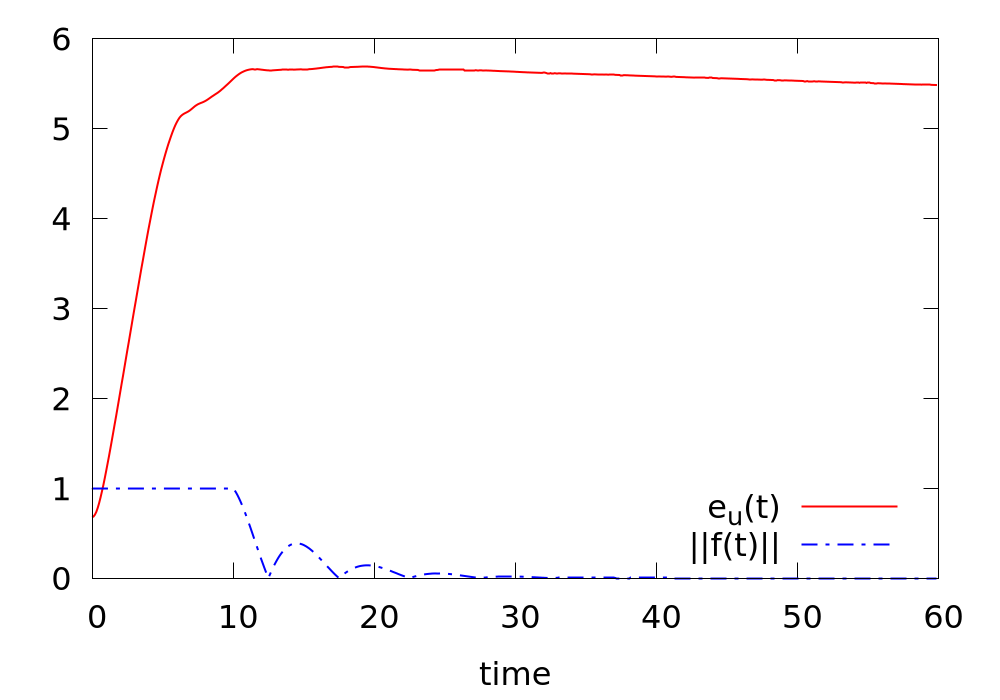}
\end{subfigure}%
\begin{subfigure}{.5\textwidth}
  \centering
  \includegraphics[width=\linewidth]{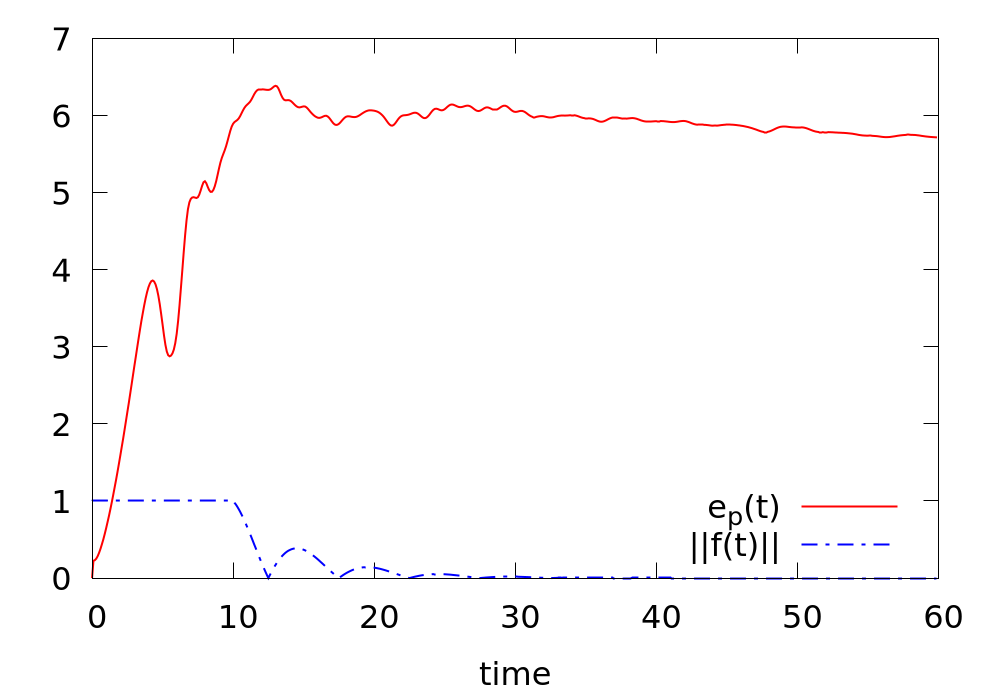}
\end{subfigure}
\caption{Time evolution of $e_\uu$ (left, red) and $e_p$ (right, red) in comparison with the norm of the external force $\ff$ (blue).}
\label{fig::time_evolution::energy}
\end{figure}
As one can see, the solution remains stable in the whole time interval $[0,60]$.
As in \cite{fengler2004} and \cite{Ganesh2010}, the initial random velocity evolves into a flow with large structures. Also the pressure develops large, smooth structures over the sphere .
In Figure \ref{fig::time_evolution::energy}, one can see that $e_\uu$ as well as $e_p$ are increasing in the time interval $[0,10]$, where the norm of the forcing term is constant at its maximum. 
As the external force decreases for $t>10$, the damping effect of the diffusion term becomes dominant and $e_\uu$ as well as $e_p$ decay slowly in time, as it can also be seen in \cite{Ganesh2010}. The results shown are thus in qualitative agreement with those of \cite{fengler2004} and \cite{Ganesh2010}.

\section{Conclusion and Remarks}
In this paper, we developed a new kernel-based collocation method for the incompressible Navier-Stokes equations on the two dimensional sphere.
We proved convergence of the method based on the fill distance of the collocation points and the smoothness of the initial data for both the velocity and the pressure of the system. 

The advantages of this method are various. 
%independend of any mesh on the sphere
The new method is a meshfree method, which means that the positions of the data points $\cX$ are independent of any underlying grid on the sphere. 
%No Quadrature formular needed
Moreover, it needs no quadrature formula for the computation, which makes it less vulnerable to additional errors.
%analytically divergence-free approximation 
The method leads to high-order analytically divergence-free approximations of the velocity field, which can be computed by any smoothness depending on the choice of the kernel function. 
% simple calculation of the pressure
By its kernel-based character, it is simple to give high-order approximations of the pressure by simply exchanging the kernel function. 
Finally, it is at least as efficient as comparable methods like the one of Fengler and Freeden \cite{fengler2004} or Ganesh, Le Gia and Sloan \cite{Ganesh2010}, as the complexity can be reduced to $\sim N^2$ calculations per time step, where $N$ is is number of data points. 

This can be achieved by using the Lagrangian basis. The Lagrange functions of the basis $\Phi(\cdot,\xx_j)\in\R^{3\times 3}$ is given by the functions
\begin{equation*}
\Psi_k=\sum_{j=1}^N\Phi(\cdot,\xx_j)\ggamma_j^k,\quad 1\leq k\leq N,
\end{equation*}
such that we have the interpolation condition $\Psi_k(\xx_j)=\delta_{k,j}I_{3\times 3}$ for every $1\leq k,j\leq N$. 
Given a Lagrange basis, the interpolation of a tangential vector-field $\uu$ is just given by
\begin{equation*}
\uu_h(\xx)=\sum_{j=1}^N\Psi(\cdot,\xx_j)\uu(\xx_j),
\end{equation*}
which means among other things that the linear systems in \eqref{Algo::eq::Euler::expl} and \eqref{Algo::eq::Euler::impl} can be given without any further calculation. However, the evaluation of the approximation $\uu_h$ becomes more expensive, since the cost of an evaluation of the function $\Psi_k$ is linear in $N$. Hence, the cost of the whole method will be $\sim N^2$. 
The use of local Lagrange functions, see for example \cite{Buhmann1990}, which reduces the cost of an evaluation of $\Psi$ at the cost of an additional error, is part of current research.

However, there remains much space to improve and investigate the current method. 
The work mainly involves the analysis of discretisation in space. An exact analysis of the time discretisation, the choice of the time discritisation scheme as well as the the time step $\tau$ and the associated stability of the method need to be researched further.

Moreover, the numerical behaviour of the method has to be further investigated.
We only validated the method on a benchmark test case with unknown analytical solution and on a nice node set.
The method has to be tested on some examples with known analytical solution to determine the numerical order of convergence and to compare them with the results from Theorem \ref{4::theroem::main_theorem} and Theorem \ref{4::theroem::main_theorem::pressure}.

Moreover, in the presented test case, only unscaled kernels where used. However, kernel functions of the form $\phi_\epsilon(x)=\phi(x/\epsilon)$ should improve the efficiency of the method since the occurring matrices will be sparser. 
However, this will be achieved at the cost of accuracy, since it is expected that the interploation error behaves at least like $\sim \epsilon^{-\sigma}$.
An investigation of the numerical error depending on $h_\cX$ and $\epsilon$ has still to be done for the presented method.

\bibliographystyle{amsplain}
\bibliography{ref.bib}

\end{document}